\documentclass[letterpaper]{amsart}

\usepackage[utf8]{inputenc}
\usepackage{lmodern}
\usepackage{microtype}
\usepackage{geometry}
\usepackage{mathtools}
\usepackage{amsthm}
\usepackage{thmtools}
\usepackage{thm-restate}
\usepackage{amsfonts}
\usepackage{amssymb}
\usepackage[titletoc,toc,title]{appendix}
\usepackage{tikz-cd}
\usepackage{xcolor}
\usepackage{todonotes}
\usepackage{bm}
\usepackage{etoolbox}

\usepackage{hyperref}
\usepackage{enumitem}

\newcommand{\Fraisse}{\text{Fra\"{i}ss\'{e}}}
\newcommand{\RCA}{\text{RCA}_0}
\newcommand{\WKL}{\text{WKL}_0}
\newcommand{\ACA}{\text{ACA}_0}

\newcommand{\Sk}{\mathrm{Sk}}

\newcommand{\tuple}[1]{\overline{#1}}
\newcommand{\seq}[2]{ (#1)_{#2} }

\newcommand{\pres}[2]{\langle #1 \mid #2 \rangle}

\newcommand{\into}{\hookrightarrow}

\newcommand{\restricted}{\upharpoonright}

\newcommand{\bland}{\bigwedge}

\newcommand{\qftp}{\mathrm{qftp}}

\newcommand{\cC}{\mathcal{C}}
\newcommand{\cK}{\mathcal{K}}

\newcommand{\cS}{\mathcal{S}}

\newcommand{\bK}{\mathbb{K}}

\newcommand{\Tdega}{\mathbf{a}}
\newcommand{\Tdegb}{\mathbf{b}}

\declaretheoremstyle[
headfont=\normalfont\bfseries,
parent=subsection,
bodyfont=\normalfont,
qed=$\dashv$,
]{def2}

\declaretheorem[style = def2, title = Definition]{defn}

\theoremstyle{definition}
\newtheorem{quest}[defn]{Question}

\newtheorem{notation}[defn]{Notation}

\theoremstyle{remark}

\newtheorem*{rem}{Remark}
\newtheorem{no-rem}[defn]{Remark}

\theoremstyle{theorem}
\newtheorem{theorem}[defn]{Theorem}

\newtheorem{cor}[defn]{Corollary}

\newtheorem{prop}[defn]{Proposition}
\newtheorem{lemma}[defn]{Lemma}

\newtheorem{obs}[defn]{Observation}

\AtEndEnvironment{rem}{$\dashv$}

\title{On Effective Constructions of Existentially Closed Groups}
\author{I Scott}
\thanks{This work forms part of the author's doctoral thesis at the University of Chicago, under the wise and generous supervision of Denis Hirschfeldt and Maryanthe Malliaris.}
\address{Department of Mathematics\\
	University of Chicago\\
	5734 S University Ave\\
	Chicago, IL 60637}
\email[I Scott]{iscott@uchicago.edu}

\begin{document}
	\maketitle
	
	\begin{abstract}
		Existentially closed groups are, informally, groups that contain solutions to every consistent finite system of equations and inequations.  They were introduced in 1951 in an algebraic context and subsequent research elucidated deep connections with group theory and computability theory.  We continue this investigation, with particular emphasis on illuminating the relationship with computability theory.  
		
		In particular, we show that existentially closed groups can be built at the level of the halting problem and that this is optimal.  Moreover, using the the theory of the enumeration degrees and some work of Martin Ziegler in computable group theory, we show that the previous result relativises in a somewhat subtle way.  We then tease apart the complexity contributed by ``global'' and ``local'' structure, showing that the finitely generated subgroups have complexity at the level of the PA degrees.  Finally, we investigate the computability-theoretic complexity of omitting the non-principal quantifier-free types from a list of types, from which we obtain an upper bound on the complexity of building two existentially closed groups that are ``as different as possible''.  
	\end{abstract}
	
	\section*{Overview of results}
	A group $M$ is called existentially closed when any finite system of equations and inequations with parameters from $M$ that is solved in a group extending $M$ already has a solution in $M$ (see Definition \ref{def of ec group}).  While the idea of existential closure had been floating around for a while --- some, e.g. \cite{Hodges-model-theory},  trace it back to the first suggestions of the complex numbers --- the definition of an existentially closed group was first proposed by W.R.\ Scott in \cite{Scott}, probably as a group-theoretic analogue to the notion of algebraic closure for fields.  Within a decade, model theorists, led by Abraham Robinson, had isolated the notion of existential closure for general theories and had developed their own abstract ways of building and studying them.  This model-theoretic perspective greatly influenced the study of existentially closed groups, as can be seen in \cite{omitting-qf-types}, \cite{on-ac-groups}, and \cite{Hodges}, among others.  From both the group- and model-theoretic perspectives, we see computability theory naturally arising as a way to stratify and measure the complexity of existentially closed groups.  This is perhaps most thoroughly illustrated by Ziegler in \cite{Ziegler}.  Ziegler characterised existentially closed groups via a computability-theoretic invariant by understanding how different finitely generated groups can fit together to form an existentially closed group.  This paper seeks to understand the computability-theoretic complexity of the existentially closed groups themselves, shedding further light on the structure of these groups, and how they relate to each other.\\
	
	We lay out here an overview of our main results for those more familiar with the language of computability theory.  More context and precise statements can be found in the next section.
	
	Our first main result is that building an existentially closed group is computability-theoretically equivalent to the halting problem (Theorem \ref{existence of ec groups is at level of hp}).  This result relativises in an unexpected way (Theorem \ref{relativisation of existence of ec group is level of hp}), involving interactions with the enumeration degrees and ideas of Ziegler relating to his relativisation of Higman's embedding theorem.  We also show that this complexity is not always reflected in the finitely generated subgroups of an existentially closed group, and quantify the complexity that they contribute (Theorem \ref{every scott set gives an ec group}), which gives us a new characterisation of the PA degrees (Theorem \ref{char of PA degrees}).  Finally, call two existentially closed groups ``$\exists_1$-atomic'' if the only quantifier-free types realised in both are isolated by an existential formula.  We show that there is such a pair both of which are computable in any $0'$-c.e.\ set that is strictly above $0'$.

	\section*{Introduction}
	In this section, we introduce the main objects of study, provide some mathematical context for the paper, and give formal statements of our main results.
	
	\begin{defn}
		\label{def of ec group}
		An existentially closed group $M$ is a group such that for every quantifier-free formula $\varphi(\tuple{x}, \tuple{m})$, where $\tuple{m} \in M$\footnote{Throughout the paper, we will write ``$\tuple{m} \in M$'' as shorthand for ``$\tuple{m} \in M^{\ell(\tuple{m})}$''.}, if there is a group $N$ containing $M$ as a subgroup with $N \models \exists \tuple{x} \varphi(\tuple{x}, \tuple{m})$, then $M \models \exists \tuple{x} \varphi(\tuple{x}, \tuple{m})$.
	\end{defn}
	
	We will be interested in the complexity of existentially closed groups.  From a computability-theoretic standpoint, a natural way to measure the complexity of an algebraic object is to look at the Turing degree of its quantifier-free diagram.  In the case of groups, this generalises the well-studied notion of the degree of the word problem, so this fits into a natural lineage of study on both the computability-theoretic and algebraic sides.  (Aside: as we will discuss in the next section, the degree of a word problem is only well-defined for finitely generated groups and, as a result, word problems of infinitely generated groups --- which existentially closed groups always are --- are not often studied in the group theory literature.  Computability theory bypasses this issue by instead studying the class of degrees that compute some listing of the group.)  The interactions between the word problems of finitely generated groups and the degrees of the existentially closed groups they generate will be a theme of the present work.
	
	\textit{As we take a computability-theoretic approach, from now on all existentially closed groups will implicitly be assumed to be countable.}\\
	
	We now lay out some of the history of existentially closed groups that forms the backdrop for the rest of the paper.  The first result is one of the earliest major theorems on existentially closed groups, which highlights how they blend algebra, computability theory, and model theory.
	
	In what follows, a finitely generated group $G$ is \emph{$\exists_1$-isolated} if there is a quantifier-free formula of group theory $\varphi(\tuple{x}, \tuple{y})$ such that whenever $H$ is a group such that $H \models \exists \tuple{x} \varphi(\tuple{x}, \tuple{a})$ for some $\tuple{a} \in H$, then $\langle \tuple{a} \rangle$ generates a copy of $G$ in $H$.
	
	\begin{theorem}[\cite{Neumann}, \cite{omitting-qf-types}, \cite{Rips}]
		\label{macintyre-neumann-rips}
		Let $G$ be a finitely generated group.  Then the following are equivalent:
		
		\begin{enumerate}
			\item $G$ embeds in every existentially closed group.
			\item $G$ has a solvable word problem.
			\item $G$ is $\exists_1$-isolated.
		\end{enumerate}
	\end{theorem}
	
	While most of the equivalences were proved using group-theoretic arguments, Macintyre (\cite{omitting-qf-types}) used the newly developed model-theoretic forcing to show that, for any finitely generated $G$ with unsolvable word problem, there is an existentially closed $M$ that does not contain a copy $G$.  To do so, he proved a much more general theorem.  In fact, the result in \cite{omitting-qf-types}, Section 3, is even more general than what we state here.  (We use $\leq_{\mathrm{T}}$ for Turing reducibility; this should cause no confusion with the theory $T$ once pointed out.)
	
	\begin{theorem}[\cite{omitting-qf-types}]
		Let $T$ be a decidable $\forall\exists$ theory and let $\Psi$ and $\Phi$ be quantifier-free $T$-types such that $\Psi \nleq_{\mathrm{T}} \Phi$.  Then there is an existentially closed $M \models T$ that realises $\Phi$ and omits $\Psi$.
	\end{theorem}
	
	It follows from this theorem that the original result relativises:  if $G$ and $H$ are finitely generated groups with word problems $W(G)$ and $W(H)$ respectively and $W(G) \nleq_{\mathrm{T}} W(H)$, then there is an existentially closed group that contains $H$ and not $G$.
	
	It is natural to ask whether the converse holds: namely if $W(G) \leq_{\mathrm{T}} W(H)$, does $G$ embed into every existentially closed group that $H$ does?  Ziegler showed the answer is ``sort of'': this holds for a somewhat strengthened notion of computability, $*$-reducibility, written $\leq^*$ (see brief discussion after the result).
	
	\begin{theorem}[\cite{Ziegler}]
		Let $G$ and $H$ be finitely generated groups with word problems $W(G)$ and $W(H)$ respectively.  Then the following are equivalent:
		\begin{enumerate}%[(a)]
			\item $W(G) \leq^* W(H)$.
			\item $G$ embeds into every existentially closed group that $H$ does.
		\end{enumerate}
	\end{theorem}
	
	While we will not need it for the present paper, the definition of $\leq^*$ can be found in \cite[Chapter III.1]{Ziegler}.  The key takeaways are:
	
	\begin{itemize}
		\item $*$-reducibility refines Turing-reducibility and enumeration-reducibility (which will be discussed further in Section \ref{subsection: enumeration reducibility}); i.e., $A \leq^* B$ implies $A \leq_{\mathrm{T}} B$ and $A \leq_{\mathrm{e}} B$.
		\item $\leq^*$ is defined purely computability-theoretically --- without reference to groups.
	\end{itemize}
	
	Thus, Ziegler's theorem shows that the property of embedding into every existentially closed supergroup of a given group is really a computability-theoretic notion.   Moreover, since countable existentially closed groups are determined up to isomorphism by their finitely generated subgroups, it veers towards a classification (and, in fact, Ziegler does show that countable existentially closed groups are in one-to-one correspondence with what he calls \emph{algebraically closed ideals} in the $*$-degrees --- see \cite{Ziegler}, Chapter III.3).
	
	Going back to Theorem \ref{macintyre-neumann-rips}, another natural question is whether there is any existentially closed group all of whose finitely generated subgroups have solvable word problems.  Macintyre, relying on a result of Miller, answered that question in the negative.
	
	\begin{theorem}[\cite{on-ac-groups}; based on \cite{Miller}]
		Let $M$ be an existentially closed group.  Then $M$ has a finitely generated subgroup with an unsolvable word problem.
	\end{theorem}
	
	Thus, in particular, no existentially closed group is computable.  
	
	In fact, a close reading of Miller's argument shows that every existentially closed group contains a finitely generated subgroup of \emph{PA degree} (see Definition \ref{PA deg-def}).  It is known that there are no computable PA degrees, but that they can be ``quite close to being computable''. 
	
	Despite the connections with computability noticed by Macintyre, Ziegler, and others, there had been no systematic investigations into the computability theory of the existentially closed groups themselves.  Indeed, to summarise the historical development outlined above: If one wants to understand countable existentially closed groups from a purely algebraic standpoint, it suffices to understand their finitely generated subgroups, as this determines the existentially closed group up to isomorphism.  The thrust of many earlier results can be understood from this point of view: at least when investigating the computability-theoretic properties of existentially closed groups, earlier authors focused on their finitely generated subgroups.
	
	In the present work, we take a step back, treating the existentially closed groups as objects of computability-theoretic study in their own right.	 Through our investigations into the computability theory of existentially closed groups, a stratification emerges that was not visible before --- the complexity arising from the finitely generated subgroups and the complexity coming from how they fit together.  \\
	
	Our first main theorem, Theorem \ref{existence of ec groups is at level of hp}, answers the question of the minimal complexity of an existentially closed group.
	
	\begin{restatable*}{theorem}{MainOne}
		\label{existence of ec groups is at level of hp}
		The existence of an existentially closed group is at the level of the halting problem; i.e., every existentially closed group computes $0'$ and there is an existentially closed group which is computed by $0'$.
	\end{restatable*}
	
	It is then shown, via a foray into the interaction between Turing degrees, enumeration degrees, and algebra, that Theorem \ref{existence of ec groups is at level of hp} relativises, although in a somewhat subtle way.  We show that the minimum degree of an existentially closed group containing a given finitely generated group $G$ is the Turing degree of $K_{W(G)}$, a set defined in Section \ref{subsection: enumeration reducibility}.  The upshots of this theorem are that this minimal degree always exists and is obtained, and that it is not merely the Turing jump of $W(G)$.  Theorem \ref{relativisation of existence of ec group is level of hp} is a more precise statement of this result.
	
	Our next main theorem illustrates a stratification in the origins of the complexity of the quantifier-free diagram of existentially closed groups, showing that Theorem \ref{existence of ec groups is at level of hp} is not a local result.  We write $\Sk(M)$ for the collection of finitely generated subgroups of $M$.
	
	\begin{restatable*}{theorem}{MainTwo}
		\label{every scott set gives an ec group}
		Let $\Tdega$ be a PA degree.  Then there is an existentially closed group $M$ such that for every $G \in \Sk(M)$, $W(G) \leq_{\mathrm{T}}\Tdega$.
	\end{restatable*}
	
	In other words, since there are PA degrees that are in a technical sense ``considerably below'' $0'$, much of the complexity of existentially closed groups is not necessarily witnessed at the finitely generated level.  Moreover, it shows that the Macintyre argument, while not capturing the full complexity of existentially closed groups, was the best that could be done by looking at their finitely generated subgroups.  
	
	From Theorem \ref{every scott set gives an ec group}, we obtain a new characterisation of PA degrees, which form an important class of Turing degrees (see Definition \ref{PA deg-def}).  Again, $\Sk(M)$ is the collection of finitely generated subgroups of $M$.
	
	\begin{restatable*}{theorem}{MainThree}
		\label{char of PA degrees}
		A Turing degree $\Tdega$ is a PA degree iff there is an existentially closed group $M_{\Tdega}$ such that $\Tdega \geq_{\mathrm{T}}W(G)$ for every $G \in \Sk(M_{\Tdega})$.
	\end{restatable*}
	
	Finally, we turn to the problem of constructing two existentially closed groups whose only common subgroups are those with solvable word problems --- which, as we have seen, embed in every existentially closed group.  On the basis of Theorem \ref{macintyre-neumann-rips}, we call a pair of such groups \emph{relatively atomic}.  We give an upper bound on the degrees of these groups, but a precise characterisation remains unknown.
	
	\begin{restatable*}{cor}{MainFour}
		\label{rel atomic groups below ce}
		Let $A \gneqq_{\mathrm{T}} 0'$ be $0'$-c.e.  Then there are $A$-computable existentially closed groups which are relatively atomic.
	\end{restatable*}
	
	The structure of the paper is as follows:
	
	Section 1 lays out the background in group theory and logic that will be assumed throughout the rest of the paper.  Sections 2 and 3 develop a construction for building existentially closed groups effectively, following work of \cite{Ziegler} and \cite{effective-fraisse}.  Then Sections 4 and 5 leverage this construction and group-theoretic arguments to establish that the existence of existentially closed groups is exactly at the level of the halting problem, and show how this result relativises.  Section 6 shifts focus to the subgroups of existentially closed groups, and shows that the results above do not extend to them.  Section 7 extends the analysis of constructing existentially closed groups to constructing pairs of them that are ``as different as possible''.
	
	\section{Background}
	To keep this paper relatively self-contained and to fix notation, we will lay out the mathematical preliminaries needed for the rest of the paper.  We assume familiarity with basic model theory, computability theory, and group theory.  For more details on these, we recommend \cite{Chang-Keisler}, \cite{Soare}, and \cite{Rotman} respectively.
	
	We will adhere to the following conventions throughout the paper:
	\begin{itemize}
		\item Existentially closed groups are implicitly assumed to be countable.
		\item Unless otherwise mentioned, ``computable'' and its derivatives will always refer to Turing computability.
		\item $L$ is the language $\{1, \cdot, \bullet^{-1}\}$ for group theory.  (Here $\bullet^{-1}$ indicates that we use the notation $g^{-1}$ for ``$g$-inverse''.)
	\end{itemize}
	
	Most notation and definitions are standard; however, we point out that we use ``computably presentable'' to refer to computable structures, rather than group presentations (see Definitions \ref{def: spectrum and degree} and \ref{def: group presentations}).
	
	\subsection{Notation for formulas}
	
	\begin{defn}
		\label{def: formulas}
		We establish the following conventions for $L$-formulas.
		\begin{itemize}
			\item An \emph{atomic formula} is a quantifier-free formula $\varphi(\tuple{x})$ which does not include conjunctions, disjunctions, or negations.  
			
			\item If $\varphi(\tuple{x}) = \text{``}\bland_{i \in I} \psi_i(\tuple{x}) \land \bland_{j \in J} \neg \psi_j(\tuple{x})\text{''}$, where $I$ and $J$ are finite, and $(\psi_i)_{i \in I}$ and $(\psi_j)_{j \in J}$ are atomic formulas, then we write
			\[ \varphi^+(\tuple{x}) = \{ \psi_i(\tuple{x}) : i \in I \} \]

			and call this the \emph{positive part} of $\varphi(\tuple{x})$ or the \emph{equations} of $\varphi(\tuple{x})$.
			
			Analogously, we write
			\[\varphi^-(\tuple{x}) = \{ \psi_j(\tuple{x}) : j \in J \} \]

			and call this the \emph{negative part} of $\varphi(\tuple{x})$ or the \emph{inequations} of $\varphi(\tuple{x})$.
			\item If $\Phi = \{ \varphi_i : i < n\}$ is a finite set of formulas, we write $\bland \Phi$ for $\bland\limits_{i < n} \varphi_i$.
			\item Let $M$ be an $L$-structure.  For a quantifier-free $\varphi(\tuple{x}, \tuple{m})$ with parameters $\tuple{m} \in M$, say the $L$-structure $N$ \emph{solves} or \emph{realises} $\varphi$ over $M$ if $N \geq M$ and $N \models \exists \tuple{x} \varphi(\tuple{x}, \tuple{m})$.  For $\tuple{n} \in N$ with $N \models \varphi(\tuple{n}, \tuple{m})$, we will say that $\tuple{n}$ \emph{solves} $\varphi(\tuple{x}, \tuple{m})$ in $N$.  \qedhere
		\end{itemize}
	\end{defn}
	
	The following simple proposition will simplify our arguments showing that a group $M$ is existentially closed.  It says that to check existential closure, it suffices to check finite systems of equations and inequations.
	
	\begin{prop}
		\label{no disjunctions for ec}
		For a group $M$, the following are equivalent:
		\begin{enumerate}
			\item $M$ is existentially closed; i.e., for every quantifier-free formula $\varphi(\tuple{x}, \tuple{m})$ with parameters from $M$, if there is a group $N \geq M$
			with $N \models \exists \tuple{x} \varphi(\tuple{x}, \tuple{m})$, then $M \models \exists \tuple{x} \varphi(\tuple{x}, \tuple{m})$.
			\item For every conjunction $\varphi(\tuple{x}, \tuple{m})$ of atomic and negated atomic formulas with parameters from  $M$, if there is a group $N \geq M$ with $N \models \exists \tuple{x} \varphi(\tuple{x}, \tuple{m})$, then $M \models \exists \tuple{x} \varphi(\tuple{x}, \tuple{m})$.
		\end{enumerate}
	\end{prop}
	
	\begin{proof}[Proof sketch]
		Clearly (1) implies (2).  For the other direction, let $\varphi(\tuple{x}, \tuple{m}) = \bigvee\limits_{i < n} \varphi_i(\tuple{x}, \tuple{m})$ be a quantifier-free formula in disjunctive normal form which is realised in some group $N \geq M$.  Thus, $N$ models at least one of the disjuncts, say $\varphi_i(\tuple{x}, \tuple{m})$.  Applying (2) to this $\varphi_i$, we obtain that $M \models \exists \tuple{x} \varphi_i(\tuple{x}, \tuple{m})$ and hence $M \models \exists \tuple{x} \varphi(\tuple{x}, \tuple{m})$.
	\end{proof}
	
	\subsection{Computable structure theory}
	
	\begin{defn}
		For an $L$-structure $M$, $\Delta^{\mathrm{qf}}(M)$ denotes the \emph{quantifier-free diagram} of $M$; i.e., for a fixed enumeration $\seq{m_i}{i < \omega}$ of $M$, $\Delta^{\mathrm{qf}}(M)$ is the collection of quantifier-free sentences of $L(\seq{m_i}{i < \omega})$ which hold in $M$.  
		
		Similarly $\Delta(M)$ refers to the \emph{full diagram} of $M$. 
	\end{defn}
	
	In this manuscript, we will focus on the computability-theoretic power of $\Delta^{\mathrm{qf}}(M)$, as discussed in the introduction.  The definitions below go through equally well for the full diagram, but will not be as relevant for the paper.
	
	In general, the computability-theoretic power of $\Delta^{\mathrm{qf}}(M)$ may depend on the choice of enumeration of $M$. 
	
	\begin{defn}
		\label{def: spectrum and degree}
		Let $M$ be an $L$-structure.
		\begin{itemize}
			\item We write $\mathop{Spec}(M)$ for $\{ \Tdega : \text{for some enumeration of $M$,} \Delta^{\mathrm{qf}}(M) \leq_{\mathrm{T}} \Tdega \}$.  If $\Tdega \in \mathop{Spec}(M)$, then we will say $M$ is \emph{$\Tdega$-computable} or \emph{$\Tdega$-computably presentable}.
			\item In particular, $M$ is \emph{computably presentable} or \emph{computable} when $0 \in \mathop{Spec}(M)$. 
			\item The \emph{(Turing) degree} of $M$ is the minimal element $\Tdega$ of $\mathop{Spec}(M)$, if it exists.  \qedhere
		\end{itemize}
	\end{defn}
	
	\begin{rem}
		Note that the terminology $\Tdega$-computably presentable conflicts with the standard use of ``$\Tdega$-presentable'' in group theory, which refers to the Turing degree of the set of relators of a presentation of the group.  We call this \emph{having an $\Tdega$-computable relating set} to avoid confusion, as discussed in Definition \ref{def: group presentations}.
	\end{rem}
	
	\subsection{Constructions in group theory}	
	
	\begin{defn}
		\label{defn: word}
		Let $C$ be a possibly infinite set of constants and let $C^{-1}$ be the set obtained by formally inverting each element of $C$.
		\begin{itemize}
			\item A \emph{word} on $C$ is an element of $(C \sqcup C^{-1})^{< \omega}$.  
			\item For a word $w$ in $n$ letters, we will write $w(\tuple{x})$ for the corresponding word in the free variables $\tuple{x}$, $|\tuple{x}| = n$.  Similarly, for constants $\tuple{g}$, $|\tuple{g}| = n$, write $w(\tuple{g})$ for the word where each letter is replaced by the corresponding element of $\tuple{g}$.
			\item A word $w$ is \emph{reduced} if it does not contain a consecutive subword of the form $c \cdot c^{-1}$ or $c^{-1} \cdot c$.  The set of all reduced words on a set $C$ forms a group with multiplication given by ``concatenation with cancellation'', called the \emph{free group on $C$}.\qedhere
		\end{itemize}
	\end{defn}
	
	Recall that a \emph{presentation} $\pres{X}{R}$ for a group $G$ consists of a set of \emph{generators} $X$ and a set $R$ of words on $X$ called \emph{relations} or \emph{relators} such that $G$ is isomorphic to a quotient of the free group on $X$ by the normal subgroup generated by $R$.  In other words, $G$ is isomorphic to the ``freest'' group all of whose elements are words in $X$ and satisfying that every $w \in R$ is the identity in $G$. (See \cite[Chapter 11]{Rotman} for a more precise definition.)

	\begin{defn}
		\label{def: group presentations}
		We establish the following conventions for group presentations.
		\begin{itemize}
			\item $G$ is \emph{finitely generated} if it has a presentation $\pres{X}{R}$ with $X$ finite, and \emph{finitely presented} if it has a presentation $\pres{X}{R}$ with both $X$ and $R$ finite.  To emphasise that a given tuple $\tuple{g}$ generates $G$, we will sometimes write $\pres{\tuple{g}}{R}$, and to emphasise that $R$ is a set of words whose letters are contained in $\tuple{g}$, we may write $R(\tuple{g})$.
			\item If $G \leq H$ and $\tuple{g} \in H$ generates $G$, we will write $G = \langle \tuple{g} \rangle$.
			\item $G$ has an \emph{$\Tdega$-computable set of relators} if there is a presentation $\pres{X}{R}$ with $X$ finite and $R \leq_{\mathrm{T}}\Tdega$.  When $\Tdega = 0$, we say that $G$ has a \emph{computable set of relators}. \qedhere
		\end{itemize}
	\end{defn}
	
	\begin{rem}
		As discussed in Definition \ref{def: spectrum and degree}, we will reserve the phrase ``$G$ is \emph{$\Tdega$-computably presentable}'' for its computability-theoretic meaning, namely that there is an enumeration of the elements of $G$ such that the quantifier-free diagram of $G$ with respect to this enumeration is $\Tdega$-computable.
	\end{rem}
	
	\begin{defn}
		\label{def: subgroups and supergroups}
		We will write $H \leq G$ to denote that $H$ is a \emph{subgroup} of $G$\footnote{Note that this notation is similar to, but distinct from, the notation for computability introduced in an earlier subsection.  Throughout, undecorated $\leq$ will refer to the subgroup relation or the ordering on the natural numbers, while decorated $\leq$ will denote a computability-theoretic reduction.}.  When $H \leq G$, call $G$ a \emph{supergroup} of $H$.
	\end{defn}
	
	We will make use of several constructions from group theory.  Throughout, we use $\tuple{g}, \tuple{h}$ to denote the concatenation of $\tuple{g}$ and $\tuple{h}$.
	
	\begin{defn}
		\label{defn: group constructions}
		Let $G = \pres{\tuple{g}}{R}$ and $H = \pres{\tuple{h}}{S}$ be finitely generated groups.
		\begin{itemize}
			\item The \emph{free product} of $G$ and $H$ is given by
			\[ G * H = \pres{\tuple{g}, \tuple{h}}{R \cup S}. \]
			\item Suppose $G$ and $H$ both embed $F = \pres{\tuple{f}}{V}$, with $\alpha: F \into G$ and $\beta: F \into H$.  Then the \emph{free product with amalgamation of $G$ and $H$ over $F$} is given by
			\[ G*_F H = \pres{\tuple{g}, \tuple{h}}{R \cup S \cup \{\alpha(f)\beta^{-1}(f) \vcentcolon f \in \tuple{f} \} }. \]
			\item If $G$ contains two copies of $F$ with an isomorphism, say $\alpha: F \cong F'$ between them, then the \emph{HNN extension of $G$ over $\alpha$} is given by
			\[ G*_{\alpha} = \pres{\tuple{g}, t}{R \cup \{tft^{-1}\alpha(f)^{-1} \vcentcolon f \in \tuple{f}\} } \]
			where $t$ is a letter not appearing in $\tuple{g}$. \qedhere
		\end{itemize}
	\end{defn}
	
	\subsection{Word problems}
	
	\begin{defn}
		For a finitely generated group $G = \pres{X}{R}$, let the \emph{word problem} of $G$ be
		\[ W(G, X) = \{w \vcentcolon w\ \text{a word on $X$ and $w = 1$ in $G$} \}. \hfill \qedhere \]
	\end{defn}
	
	\begin{rem}
		Note that $\pres{X}{W(G, X)}$ is always a presentation of $G$.
	\end{rem}
	
	We can somewhat explicitly describe the word problem of a group from its presentation.
	
	\begin{prop}
		\label{enumeration of G from presentation}
		Let $G = \pres{X}{R}$ be a group.  Then $W(G, X)$ consists of all conjugates of cyclic permutations of $R$ by elements of $G$.
	\end{prop}
	
	This is a standard fact from group theory, which implies that, for example, any finitely generated group with a c.e.\ set of relators has c.e.\ word problem.
	
	The following easy proposition justifies the use of the phrase \emph{the} degree of the word problem for $G$.
	
	\begin{prop}
		Let $G$ be finitely generated.  Then the degree of $W(G)$ is independent of the choice of finite generating set for $G$, and, moreover, is equal to the minimal degree of $\mathop{Spec}(G)$ (which exists).
	\end{prop}
	
	Hence, from now on, we write $W(G)$ for $W(G, X)$.
	
	Thus the degree of a group generalises the degree of the word problem to possibly infinitely-generated groups.  Furthermore, every Turing degree is the degree of the word problem of a finitely generated group.
	
	\begin{prop}
		\label{every turing degree contains a group}
		For every $X$, there is a finitely generated group $G$ such that $W(G) \equiv_{\mathrm{T}} X$.
	\end{prop}
	
	We will need a somewhat stronger version of this result, which can be found in \cite{Ziegler2}, and which we will see as Theorem \ref{every star degree contains a group}.

	\subsection{Higman's embedding theorem}
	
	Higman's remarkable Embedding Theorem (\cite{Higman}) provided one of the first real bridges between group theory and computability theory.  It will play a large role in many of the arguments in this paper.
	
	\begin{theorem}[\cite{Higman}]
		\label{higman's embedding theorem}
		Let $G$ be a finitely generated group.  Then $G$ has a computable set of relators iff $G$ embeds in a finitely presented group.
	\end{theorem}
	
	\begin{no-rem}
		\label{Craig}
		Higman's theorem extends to finitely generated groups with computably enumerable presentations: suppose $G$ is such a group with c.e.\ list of relations $\seq{R_i(\tuple{g})}{i < \omega}$.  Then $G$ is computably presented via $\seq{R_i(\tuple{g})e^i}{i < \omega}$.  
		
		Indeed, to check if some word is a relator in the new relating set, first check that it is of the form $w(\tuple{g})e^i$, where $w$ has no terminal $e$s.  Then, to check if this is a relator, go through the computable enumeration of the old relating set up to position $i$.  This will be enough to determine if $w(\tuple{g})e^i$ corresponds to a relator in the old relating set.
		
		This is called Craig's Trick.
	\end{no-rem}

	\section{Building existentially closed groups via the $\Fraisse$ construction}
	In this section we show that every existentially closed group is a $\Fraisse$ limit and characterise the $\Fraisse$ classes that yield existentially closed groups.  All of the results of this section, unless otherwise stated, are from \cite[Chapter I]{Ziegler}.  This characterisation of \Fraisse\ classes giving existentially closed groups, together with a relativisation of \Fraisse's Theorem, will be key ingredients in many arguments in this paper, so we include details for completeness.
	
	\begin{defn}
		$M$ is \emph{$\omega$-homogeneous} if whenever $A$ and $B$ are finitely generated substructures of $M$ and $\alpha$ is an isomorphism between $A$ and $B$, there is an automorphism $\hat{\alpha}$ of $M$ extending $f$.
	\end{defn}
	
	Note that some authors refer to this as \emph{strongly $\omega$-homogeneous}, but we will not need this distinction.
	
	\begin{prop}
		Existentially closed groups are $\omega$-homogeneous.
	\end{prop}
	
	The proposition follows by applying \textit{HNN extensions}, which show that any isomorphism between finitely generated subgroups can be realised by conjugation in a supergroup.
	
	\begin{defn}
		The \emph{skeleton} of a structure $M$, denoted $\Sk(M)$, is the collection of all finitely generated substructures of $M$.  (Note that many authors call this the \emph{age} of $M$.)
	\end{defn}
	
	A standard back-and-forth argument shows that countable $\omega$-homogeneous structures are isomorphic.  Thus we obtain the following corollary.
	
	\begin{cor}
		Countable existentially closed groups are determined up to isomorphism by their skeletons.
	\end{cor}
	
	We will provide a characterisation, due to Ziegler, of the classes of finitely generated groups that form the skeleton of an existentially closed group in Theorem \ref{characterisation of skeleton of ec group}.  This will ultimately give us an effective way of building existentially closed groups from an appropriate class of finitely generated subgroups.  This tool will be an important component of many of our later results.
	
	Before we state the characterisation, we need to define some properties of skeletons.
	
	\begin{defn}
		\label{properties of skeletons}
		Let $T$ be a theory and let $\cK$ be a collection of finitely generated models of $T$.  Then $\cK$ has the:
		\begin{itemize}
			\item \emph{Hereditary Property} (HP) if for every $A \in \cK$, and every finitely generated substructure $B$ of $A$, there is some $\hat{B} \in \cK$ such that $\hat{B} \cong B$.
			\item \emph{Joint Embedding Property} (JEP) if for every $A$, $B \in \cK$, there is a $C \in \cK$ with embeddings $\alpha: A \into C$, $\beta: B \into C$.
			\item \emph{Amalgamation Property} (AP) if for every $A$, $B \in \cK$ and every finitely generated $C$ such that there are embeddings $\alpha: C \into A$ and $\beta: C \into B$, there is a $D \in \cK$ along with embeddings $\hat{\alpha}: A \into D$ and $\hat{\beta}: B \into D$ such that $\hat{\alpha}\alpha \restricted C = \hat{\beta}\beta \restricted C$.
		\end{itemize}
		In addition, in this paper we say $\cK$ is \emph{existentially closed} if for every quantifier-free formula $\varphi(\tuple{x}, \tuple{a})$ with parameters from some $A \in \cK$ and such that $\varphi$ is solved in some superstructure $B$ of $A$ with $B \models T$, then there is a $\hat{B} \in \cK$ with $\hat{B}$ a superstructure of $A$ and $\tuple{b} \in \hat{B}$ with $\hat{B} \models \varphi(\tuple{b}, \tuple{a})$.
	\end{defn} 
	
	\begin{theorem}[\cite{Ziegler}]
		\label{characterisation of skeleton of ec group}
		A collection of finitely generated groups is the skeleton of an existentially closed group iff it is existentially closed and satisfies HP, JEP, and AP.
	\end{theorem}
	
	The proof of this theorem relies on $\Fraisse$'s Theorem, which we state below.  In the next section, we will discuss how $\Fraisse$'s Theorem can be effectivised, which will provide us with a bound on the computational power needed to build existentially closed groups.
	
	\begin{theorem}[$\Fraisse$'s Theorem]
		\label{Fraisse}
		A countable class of finitely generated structures $\cK$ is the skeleton of an $\omega$-homogeneous structure, called the \emph{$\Fraisse$ limit} of $\cK$, iff it satisfies HP, JEP, and AP.
	\end{theorem}
	
	Further detail on $\Fraisse$'s Theorem and its variations can be found in most standard model theory textbooks; for example \cite{Hodges-model-theory}, Chapter 7.1.
	
	\begin{proof}[Proof of Theorem \ref{characterisation of skeleton of ec group}]
		It is clear that the skeleton of an existentially closed group $M$ is existentially closed and satisfies HP, JEP, and AP.
		
		For the other direction, suppose $\cK$ is an existentially closed collection of finitely generated groups satisfying HP, JEP, and AP.  By $\Fraisse$'s Theorem, there is an $\omega$-homogeneous group $M$ with $\Sk(M) = \cK$.  Now it suffices to show that $M$ is existentially closed.  
		
		Consider $\exists \tuple{x} \varphi(\tuple{x}, \tuple{g})$, where $\varphi$ is quantifier-free and $\tuple{g} \in G \in \Sk(M)$.  Suppose in addition that $\varphi$ is solved in some $N \geq M$.  Since $G \in \Sk(M)$, it follows that $N \geq G$.  Thus, by existential closure of $\cK$, there is some $H \in \cK$ with $H \geq G$ and such that $\varphi$ is solved in $H$.  But $H \leq M$, since $\Sk(M) = \cK$, so $\varphi$ is solved in $M$, as required. 
	\end{proof}
	
	\begin{rem}
		This proof goes through equally well for other theories, so every existentially closed \Fraisse\ class gives rise to an $\omega$-homogeneous existentially closed structure and every $\omega$-homogeneous existentially closed structure determines an existentially closed $\Fraisse$\ class.
	\end{rem}
	
	\section{Building existentially closed groups effectively}
	In this section, we present the second preliminary construction needed for our work --- a computable version of $\Fraisse$'s Theorem (Theorem \ref{effective fraisses theorem}) due to \cite{effective-fraisse} --- and show that it relativises.  This will allow us to effectively construct existentially closed groups with various properties.  
	
	In order to state the effective $\Fraisse$'s Theorem, we need to effectivise the notion of skeleton, and properties that a skeleton may satisfy, from above.
	
	\begin{defn}
		A sequence $\bK = \seq{A_i, \tuple{a_i}}{i<\omega}$ is a \emph{representation} of a collection $\cK$ of finitely generated structures if $\cK = \{A_i : i < \omega\}$ (up to isomorphism), $\tuple{a_i}$ is a finite tuple, $A_i$ is generated by $\tuple{a_i}$, and the domain of each $A_i$ is a subset of $\omega$.
		
		A representation is \emph{computable} if the sequence $\seq{\tuple{a_i}}{i < \omega}$ is computable and the functions, relations, and constants of $\seq{A_i}{i < \omega}$ are uniformly computable.
		
		A \emph{computable skeleton} is a computable representation of a skeleton. \qedhere
	\end{defn}

	\begin{defn}
		A computable skeleton $\bK = \seq{A_i, \tuple{a_i}}{i<\omega}$ has the \emph{computable extension property}, or \emph{computable-EP}, if there is a partial computable function that takes a pair $(i, \varphi(\tuple{x}, \tuple{a}))$, where $\varphi$ is quantifier-free with parameters from $A_i$, and returns a structure $(A_j, \tuple{a_j})$, an embedding $\alpha: A_i \into A_j$, and a tuple $\tuple{b} \in A_j$ such that $A_j \models \varphi(\tuple{b}, \alpha(\tuple{a}))$ if such a structure exists, and that does not halt otherwise. 
		
		If there is such $(A_j, \tuple{a_j})$ and $\tuple{b}$, say $\varphi(\tuple{x}, \tuple{a})$ is \emph{consistent with $A_i$ in $\bK$}. \qedhere
	\end{defn}
	
	Note that $\bK$ is existentially closed in the sense of Definition \ref{properties of skeletons} exactly when, for every $i$ and $\varphi(\tuple{x}, \tuple{a})$ with $\tuple{a} \in A$, $\varphi$ is consistent with $A_i$ in $\bK$ iff $\varphi$ is consistent with $A_i$.  In light of this, when $\bK$ is existentially closed, we say $\bK$ has \emph{computable existential closure}, or \emph{computable EC} instead of the computable extension property.\\
	
	We are now in a position to state an effectivisation of $\Fraisse$'s Theorem due to Csima--Harizanov--Miller--Montalb\'an.  (Although they prove other effectivisations, we will refer to this one as \emph{the effective $\Fraisse$'s Theorem} in this paper.)
	
	\begin{theorem}[{Effective $\Fraisse$'s Theorem, \cite[Theorem 3.12]{effective-fraisse}}]
		\label{effective fraisses theorem}
		Let $\bK$ be a computable skeleton with HP, JEP, and AP.  Then $\bK$ has a computable $\Fraisse$ limit if and only if it has a computable representation with the computable extension property.
	\end{theorem}
	
	The relativisation of this theorem will be our primary method for effectively constructing existentially closed groups.  It follows by relativising the proof of the effective $\Fraisse$ Theorem in \cite{effective-fraisse}.
	
	\begin{defn}
		\label{def of rel fraisse props}
		Let $\Tdega$ be a Turing degree.  $\bK$ is an \emph{$\Tdega$-computable skeleton} if there is an $\Tdega$-computable enumeration of $\bK$ such that the constants, functions, and relations are uniformly $\Tdega$-computable.
		
		$\bK$ has the \emph{$\Tdega$-computable extension property}, written $\Tdega$-EP, (or \emph{$\Tdega$-computable existential closure}, written $\Tdega$-EC, if $\bK$ is existentially closed) if there is a partial $\Tdega$-computable function $f$ that takes an index $i$ for $A_i \in \bK$ and a quantifier-free $\varphi(\tuple{x}, \tuple{a})$, where $\tuple{a} \subseteq A$, and returns an index $j$ for $A_j \in \bK$ and a $\tuple{b} \subseteq A_j$ such that $A_j \models \varphi(\tuple{b}, \tuple{a})$, if such a $j$ and $\tuple{b}$ exist.  If no such $j$ and $\tuple{b}$ exist, $f$ does not halt.
	\end{defn}
	
	\begin{theorem}
		\label{rel effective fraisses theorem}
		Let $\textbf{a}$ be a Turing degree and $\bK$ an $\Tdega$-computable skeleton with HP, JEP, and AP.  Then $\bK$ has an $\Tdega$-computable $\omega$-homogeneous $\Fraisse$ limit if and only if it has the $\Tdega$-EP.
	\end{theorem}
	
	\section{An existentially closed group in the halting problem}
	In this section, we prove our first main result characterising the minimal Turing degree necessary for constructing existentially closed groups.
	
	\MainOne
	
	We prove each half of the result separately over the next two subsections.
	
	\subsection{Every existentially closed group computes $0'$}
	\label{section: every ec group computes hp}
	The main result of this subsection is the following theorem.
	
	\begin{theorem}
		\label{every ec group computes hp}
		Let $M$ be an existentially closed group.  Then $\Delta^{qf}(M) \geq_{\mathrm{T}} 0'$.
	\end{theorem}
	
	In order to prove this, we require two group-theoretic lemmas.
	
	\begin{lemma}
		\label{existence of fp group computing hp}
		There is a finitely presented group $G$ such that $W(G) \equiv_{\mathrm{T}} 0'$.
	\end{lemma}
	
	\begin{proof}[Proof Sketch]
		Consider the group
		\[ H = \pres{a,b}{\{[a,b^nab^{-n}] : n \in \emptyset'\}}. \]
		
		Using that $\{b^nab^{-n} : n < \omega\}$ freely generate a copy of $F_{\omega}$ --- the free group on countably infinitely many generators --- in $F_2$ and that $\emptyset'$ is c.e., one can show that $W(H) \equiv_{\mathrm{T}} 0'$.
		
		Since $\emptyset'$ is c.e., Craig's trick (see Remark \ref{Craig}) and Higman's embedding theorem imply that $H$ embeds in a finitely presented group $F$.  Then, since $H \leq F$ and $H$ is finitely generated, $W(F) \geq_{\mathrm{T}} 0'$ and, since $F$ is finitely presented, $W(F) \leq_{\mathrm{T}} 0'$.
		
		(In fact, Clapham has shown that the finitely presented group $F$ guaranteed by Higman's theorem can always be chosen to have the same Turing degree as the original computably presented group \cite{Clapham}.)
	\end{proof}
	
	The second will allow us to encode the word problem of any finitely presented group into any existentially closed group.  It may be useful to recall the notation $w(\tuple{x})$ from Definition \ref{defn: word}.
	
	\begin{lemma}
		Let $G = \langle \tuple{g} \mid R_0(\tuple{g}), \ldots, R_{m-1}(\tuple{g}) \rangle$ be a finitely presented group.  Define $$\varphi_w^G(\overline{x}) = \text{``} R_0(\overline{x}) = 1 \land \cdots \land R_{m-1}(\overline{x}) = 1 \land w(\overline{x}) \neq 1 \text{''}$$ where $|\tuple{x}| = |\tuple{g}|$.
		
		Then $\exists \tuple{x} \varphi_w^G(\overline{x})$ is consistent iff $w(\tuple{g}) \neq 1$ in $G$.
	\end{lemma}
	
	\begin{proof}
		If $w(\tuple{g}) = 1$ in $G$, then it must be the consequence of the relations $R_0, \dots, R_{m-1}$.  So, whenever all these relations hold, $w = 1$, so $\exists\tuple{x}\varphi_w^G$ is inconsistent.
		
		On the other hand, if $w(\tuple{g}) \neq 1$ in $G$, then $\varphi_w^G$ is realised in $G$ and hence is consistent.
	\end{proof}
	
	We are now in a position to prove the main theorem of this section.
	
	\begin{proof}[Proof of Theorem \ref{every ec group computes hp}]
		Assume $G$ is a finitely presented group with a word problem of degree $0'$ and let $M$ be any existentially closed group.  Let $W(G)$ be the word problem of $G$ and define $\varphi_w^G(\overline{x})$ as above.  By existential closure, each $\varphi_w^G$ must be realised in $M$ iff it is consistent.  
		
		We will show that $W(G)$ and $W(G)^c$ (i.e., the complement of $W(G)$) are both $M$-computably enumerable.  $W(G)$ is already c.e., so this direction is done.
		
		On the other hand, for a fixed word $w(\overline{x})$ and a given enumeration $(m_0, m_1, \ldots )$ of $M$, the procedure that tests whether $|\overline{x}|$-length strings from $M$ satisfy $\varphi_w^G(\overline{x})$ terminates iff $\varphi_w^G$ is consistent, and hence iff $w(\overline{g}) \neq 1$ in $G$.
	\end{proof}
	
	\subsection{$0'$ computes an existentially closed group}
	In this section, we construct an existentially closed group $M_0$ whose quantifier-free diagram is computed by $0'$.  We will do this in two parts: we first build $M_0$ non-effectively following \cite{Ziegler}, and then use the machinery of the effective $\Fraisse$'s Theorem to show that it can be built in $0'$.
	
	\begin{theorem}[\cite{Ziegler}]
		\label{def of M_0}
		There is an existentially closed group $M_0$ whose skeleton consists of precisely those finitely generated groups with computable sets of relators.
	\end{theorem}
	
	We will see later how $M_0$ fits into a larger class of existentially closed groups.
	
	\begin{proof}
		Let $\cK_0$ be the collection of finitely generated groups with computable sets of relators.  By Theorem \ref{characterisation of skeleton of ec group}, it suffices to show that $\cK_0$ is existentially closed and has HP, JEP, and AP.  (See Definition \ref{properties of skeletons}).
		
		For HP, let $G = \pres{\tuple{g}}{R}$ be a finitely generated group with computable set of relators, and let $H \leq G$ with $H$ generated by $\tuple{h} = (h_0, \ldots, h_{n-1})$ in $G$.  Writing $h_i = w_i(\tuple{g})$, we obtain a new presentation $G = \pres{\tuple{g}, \tuple{h}}{R \cup \{w_i(\tuple{g})h_i^{-1} : i < n\}}$ which still has a computable relating set.  Using Proposition \ref{enumeration of G from presentation}, we get a computable enumeration of $W(G, \tuple{g} \cup \tuple{h})$.  Looking at the subsequence consisting only of words in $\tuple{h}$, we get that $W(H, \tuple{h})$ is c.e.  Thus, $\pres{\tuple{h}}{W(H, \tuple{h})}$ gives a presentation of $H$ with a c.e.\ set of relators.  A presentation of $H$ with a computable set of relators can be obtained using Craig's Trick, outlined in Remark \ref{Craig}.  (HP also follows more easily and less elementarily follows from Higman's embedding theorem (Theorem \ref{higman's embedding theorem})). 
		
		Suppose $G$ and $H$ are in $\cK_0$ and $F \in \cK_0$ embeds in both $G$ and $H$ .  Then $G *_F H$ has a computable set of relators (as can be seen in Definition \ref{defn: group constructions}).  This gives JEP and, for $F = \{ 1 \}$, AP.
		
		Finally, to show existential closure, let $G \in \cK_0$ with computable presentation $\pres{\tuple{g}}{R}$ and let $\varphi(\tuple{x}, \tuple{a})$ be a quantifier-free formula with parameters from $G$ that solved in some group $H \geq G$.  Say $H \models \varphi(\tuple{h}, \tuple{a})$ and, recalling Definition \ref{def: formulas}, write $\varphi^+$ for the positive part of $\varphi$ and $\varphi^-$ for the negative part of $\varphi$.  Consider the group $F = \langle \tuple{h}, \tuple{g} \mid R(\tuple{g}) \cup \varphi^+(\tuple{h}, \tuple{a}) \rangle$.  Since $F$ clearly is finitely generated with a computable set of relators, we just need to show that it solves $\varphi$.
		
		It is clear that $\overline{h}$ solves $\bland\varphi^+(\tuple{x}, \tuple{a})$ in $F$.  On the other hand, any inequation in $F$ is a consequence of the relations $R(\tuple{g}) = 1$ and $\bland \varphi^+(\tuple{h}, \tuple{a})$.  But all these equations hold in $H$, so if any inequation of $\varphi^-$ were to hold in $F$, it would also hold in $H$, contradicting our assumption that $H \models \varphi(\tuple{h}, \tuple{a})$.
	\end{proof}
	
	Now that we know that $M_0$ exists, we can prove the main result of this section; namely that $0'$ computes the quantifier-free diagram of $M_0$.  To do this, we apply the $0'$-effective $\Fraisse$'s Theorem (Theorem \ref{rel effective fraisses theorem}).
	
	\begin{theorem}
		\label{M_0 is computable in hp}
		$M_0$ is $0'$-computably presentable.
	\end{theorem}

	\begin{proof}
		As before, let $\cK_0 = \{G :\text{$G$ has a computable set of relators} \}$.  By the relativised effective $\Fraisse$'s Theorem, it suffices to list the finitely generated groups with computable relating set so that their function and constant symbols are uniformly $0'$-computable, the sequence of generators is computable, and such that it satisfies $0'$-EC.  
		
		Let $A_i = \pres{\tuple{a}_i}{R_i}$, $i < \omega$, be a computable enumeration of all computable presentations whose generating sets are an initial segment of $\omega$.  

		In order to put this into the context of Theorem \ref{rel effective fraisses theorem}, we need to obtain the quantifier-free diagram of each group in the sequence.  Since each group has a computable relating set, Proposition \ref{enumeration of G from presentation} implies the $W(G)$, and hence the quantifier-free diagram of $G$, has c.e.\ degree.  Hence, this can be done with an oracle for $0'$.
		
		Thus, $\mathbb{K}_0 = \seq{A_i, \tuple{a}_i}{i < \omega}$ is $0'$-computable enumeration of $\cK_0$ such that the quantifier-free diagrams of the groups are uniformly $0'$-computable.

		We will now show that $\bK_0$ has $0'$-EC.  Let $\varphi(\tuple{x}, \tuple{a})$ be a quantifier-free formula with parameters from $A_i$ and which is solved in some $G \geq A_i$.  We will show that there is a $0'$-procedure for finding a $j$ such that $A_j \geq A_i$ and $A_j \models \exists \tuple{x} \varphi(\tuple{x}, \tuple{a})$.  
		
		First note that by writing the finitely many elements of $\tuple{a}$ in terms of the generators $\tuple{a}_i$, and possibly adding conjuncts of the form ``$a_i^k = a_i^k$'', we may transform $\varphi$ into an equivalent formula with parameters consisting exactly of $\tuple{a}_i$.  Thus, without loss of generality, we write $\varphi = \varphi(\tuple{x}, \tuple{a}_i)$.
		
		By assumption, we can computably obtain from $i$ a computable presentation $\pres{\tuple{a}_i}{R_i}$ for $A_i$.  Consider the presentation $\pres{\tuple{a}_i, \tuple{t}}{ \varphi^+(\tuple{t}, \tuple{a}_i) \cup R_i}$, where $\tuple{t}$ consists of the first natural numbers not in $\tuple{a}_i$.  Set $S = \varphi^+(\tuple{t}, \tuple{a}_i) \cup R_i$.  Since $S$ is computable, there is some $\ell$ such that $\pres{\tuple{a}_i, \tuple{t}}{S}$ is the $\ell$th presentation in the list.  $\ell$ can be found computably in $0'$ by asking, for each  presentation $\pres{\tuple{a}_k}{R_k}$ with the right generators, whether there is a word which is in $R_k$ and not in $S$ or vice versa.  
		
		We will show that $A_{j} \models \exists \tuple{x} \varphi(\tuple{x}, \tuple{a}_i)$.  Indeed, let $\tuple{g}$ solve $\varphi$ in $G$.  Then $\langle \tuple{g}, \tuple{a}_i \rangle$ is a quotient of $A_{j}$, since it satisfies all the relations $R_j$.  But every formula in $\varphi^{-}(\tuple{x}, \tuple{a}_i)$ is satisfied by $\tuple{g}$ in $G$ and hence, since taking a quotient forces more words to be the identity, $\tuple{t}$ also solves $\varphi(\tuple{x}, \tuple{a}_i)$ in $R_{j}$.
		
		Since $\langle \tuple{a}_i \rangle$ generates $A_i$ in $G$, a similar argument also shows that $A_i \leq A_j$.
	\end{proof}

	Putting together the main results of this and the previous subsections, Theorem \ref{every ec group computes hp} and Theorem \ref{M_0 is computable in hp}, we complete the proof of Theorem \ref{existence of ec groups is at level of hp} that the existence of existentially closed groups is exactly at the level of $0'$.

	\section{Existentially closed supergroups}	
	\label{existentially closed supergroups}
	We now turn to the question of the complexity of building an existentially closed group \emph{over} a given finitely generated one.  Again, we are able to exactly identify the minimal Turing degree of an \emph{existentially closed supergroup} of $G$.  While the result and proof are a relativisation of Theorem \ref{existence of ec groups is at level of hp}, the details raise subtleties that were not covered in the unrelativised case, in particular regarding the interaction of the enumeration degrees and the Turing degrees.
	
	\begin{defn}
		Let $G$ be a finitely generated group.  An existentially closed group $M$ is an \emph{existentially closed supergroup} of $G$ iff $G$ embeds into $M$. \qedhere
	\end{defn}
	
	Although this definition works equally well for \emph{uncountable} existentially closed groups, as before, we restrict our attention to countable groups.
	
	There will always be many existentially closed groups extending a particular $G$.  However, the main theorem of this section is a relativisation of Theorem \ref{existence of ec groups is at level of hp} characterising, for every finitely generated $G$, the minimal degree of an existentially closed supergroup of $G$. Note that even the existence of a minimal such degree is not a priori obvious.  
	
	The optimal Turing degree is the degree of a set we call $K_{W(G)}$ (see Definition \ref{def of Gamma and L}), which we will see is the maximal Turing degree in the enumeration degree of $W(G)$.
	
	\begin{theorem}
		\label{relativisation of existence of ec group is level of hp}
		Let $G$ be finitely generated and let $\Tdega$ be the Turing degree of $K_{W(G)}$.  Then $\Tdega$ is exactly the level of the existence of an existentially closed supergroup of $G$.
	\end{theorem}
	
	As in the previous section, this requires two lemmas: one showing that $\Tdega$ is necessary --- i.e., every existentially closed supergroup of $G$ computes $\Tdega$; and one showing $\Tdega$ is sufficient --- i.e., $\Tdega$ computes an existentially closed supergroup of $G$.
	
	While the result for arbitrary $G$ is much more complicated than for $G = \{e\}$, we note that one direction relativises ``as expected''.
	
	\begin{theorem}
		\label{jump of wp computes ec supergroup}
		Let $G$ be finitely generated.  Then $W(G)'$ computes an existentially closed supergroup of $G$.
	\end{theorem}
	
	However, this is not optimal in general: $M_0$, the $\Fraisse$ limit of the computably presentable groups, is $0'$-computable and contains finitely generated subgroups of degree $0'$.
	
	\begin{notation}
		For the remainder of this section, fix a finitely generated group $G$.
	\end{notation}
	
	In the next subsection, we lay out the requisite background in \emph{enumeration degrees}, which we will see are interwoven with the structure of the \emph{finitely presented extensions} of $G$.  In the following subsection, we introduce the relativisations of the relevant group theoretic notions.  Then, we combine these ideas in the subsequent subsections to prove the main result.
	
	%	Characterising the minimal degree of an existentially closed supergroup requires the language and theory of enumeration degrees, as we see in subsection \textcolor{red}{??}.  The proof will also require relativisations of the corresponding arguments in the previous sections; we believe that many of the relativised concepts and results are not obvious and may be of independent interest.  We gather the key new components of the relativisation in the following subsection.
	
	\subsection{Enumeration Reducibility}
	\label{subsection: enumeration reducibility}
	The notion of enumeration reducibility will be key to our understanding of the computational power of the existence of an existentially closed supergroup of $G$.  We review some background on enumeration reducibility here.  More information can be found in, for example, \cite[Chapter 2]{Soskova}.
	
	Intuitively, $Y \leq_{\mathrm{e}} X$ if any enumeration of $X$ yields an enumeration of $Y$.  The following definition makes this precise.
	
	\begin{defn}
		\label{enumeration red: def}
		Let $\seq{W_k}{k < \omega}$ be a computable listing of the c.e.\ sets.  
		
		Let $X$, $Y \subseteq \omega$ and let $\seq{D_i}{i < \omega}$ be a computable enumeration of the finite subsets of $\omega$.  Then $Y$ is \emph{enumeration reducible} to $X$, written $Y \leq_{\mathrm{e}} X$, iff there is some $k \in \omega$ such that for all $n \in \omega$,
		\[n \in Y \quad \text{if and only if}\quad \exists i(\langle n,i \rangle \in W_k \land D_i \subseteq X).\]
		
		This induces an equivalence relation on $2^{\omega}$ denoted by $Y \equiv_{\mathrm{e}} X$.  The equivalence classes are called \emph{enumeration degrees}.
	\end{defn}
	
	Two notions of relative computation that we have used thus far --- being Turing reducible to a set $X$ or being computably enumerable in a set $X$ --- can both be defined in terms of enumeration reducibility.
	
	\begin{prop}
		Let $X, Y \subseteq \omega$.  Then:
		\begin{itemize}
			\item $Y \leq_{\mathrm{T}}X$ iff $Y \oplus Y^c \leq_{\mathrm{e}} X \oplus X^c$.
			\item $Y$ is c.e.\ in $X$ iff $Y \leq_{\mathrm{e}} X \oplus X^c$.
		\end{itemize}
	\end{prop}
	
	This motivates the following classical definition, which isolates a copy of the Turing degrees in the enumeration degrees.
	
	\begin{defn}
		A set $X$ is \emph{total} if $X \equiv_{\mathrm{e}} X \oplus X^c$, or equivalently if $X^c \leq_{\mathrm{e}} X$.  An enumeration degree $\Tdegb$ is \emph{total} if it contains a total set.
	\end{defn}	
	
	We will be interested in the subsets of $\omega$ that are enumeration reducible to a given set $X$.  
	
	\begin{defn}
		\label{def of Gamma and L}
		For a fixed computable listing of the c.e.\ sets $\seq{W_k}{k < \omega}$, define the \emph{enumeration operator}
		\[\Gamma_k(X) = \{n : \exists m \left( \langle n, m \rangle \in W_k \land D_m \subseteq X \right) \}.\]
		
		Furthermore, we write $K_X := \bigoplus_{k < \omega} \Gamma_k(X)$.
	\end{defn}
	
	Note that as $k$ varies in the natural numbers, $\Gamma_k(X)$ lists all the sets $Y \leq_{\mathrm{e}} X$.  Thus the following proposition is evident and implies that $K_X$ gives the maximal Turing degree that intersects the enumeration degree of $X$.  Moreover, $K_X$ is Turing-bounded above by $X'$.
	
	\begin{prop}
		\label{L is Turing maximal}
		Let $X$ be a set.  Then:
		\begin{itemize}
			\item $K_X \equiv_{\mathrm{e}} X$.
			\item $K_X \geq_{\mathrm{T}} Y$ for every $Y \leq_{\mathrm{e}} X$.
			\item $K_X \leq_{\mathrm{T}} X'$.
		\end{itemize}
	\end{prop}	
	
	As in the Turing degrees, there is a notion of a jump associated with the enumeration degrees.  Enumeration jumps were introduced, in a slightly different form, in \cite{Cooper}.
	
	\begin{defn}
		Define the \emph{enumeration jump} of $X$ by 
		\[ J_e(X) = K_X \oplus (K_X)^c. \qedhere \]
	\end{defn}
	
	Note that this definition immediately gives us that $J_e(X)$ is total for every $X$.  We continue to write $X'$ for the Turing jump of $X$.
	
	\begin{prop}\footnote{The author wishes to thank Mariya Soskova for pointing out this proposition.}
		\label{jumps of total sets a T equivalent}
		Let $X$ be total.  Then $K_X \equiv_1 X'$.  In particular, $J_e(X) \equiv_{\mathrm{T}} X'$.
	\end{prop}
	
	\begin{proof}
		Since $X$ is total, $Y$ is computably enumerable in $X$ iff $Y \leq_{\mathrm{e}} X$.  Moreover, going from the c.e.\ index to the enumeration index is uniform, so $K_X \geq_1 Y$ for every $Y$ c.e.\ in $X$. Thus, $K_X$ is 1-complete and hence $K_X \equiv_1 X'$.
	\end{proof}
	
	The following result, a special case of Theorem 1.2 in \cite{Soskov}, shows that the enumeration jump operation satisfies a jump inversion to total sets\footnote{Again, the author thanks Mariya Soskova for making them aware of this result.}.
	
	\begin{theorem}[\cite{Soskov}]
		\label{enumeration jump inversion}
		Let $Y$ be a set.  Then there is a total $X \geq_e Y$ with $J_e(X) \equiv_{\mathrm{e}} J_e(Y)$.
	\end{theorem} 
	
	(This is obtained from the corresponding result in \cite{Soskov} by setting $B_0 = B_1 = Y$, and $Q = J_e(Y)$.)
	
	As we will be working with enumeration degrees of word problems, we finish with a couple of results showing that word problems and enumeration degrees interact nicely.  
	
	\begin{prop}
		\label{e-degree of wp is well-defined}
		Let $G$ be a finitely generated group.  Then the enumeration degree of $W(G, \tuple{g})$ does not depend on the choice $\tuple{g}$ of generators of $G$.
	\end{prop}
	
	The proof of this proposition is simple and relies on $G$ being finitely generated.
	
	Finally, the following result of Ziegler improves Proposition \ref{every turing degree contains a group}.
	
	\begin{theorem}[\cite{Ziegler2}]
		\label{every star degree contains a group}
		For every set $X$, there is a finitely generated group $G$ such that $W(G) \equiv_{\mathrm{T}} X$ and $W(G) \equiv_{\mathrm{e}} X$.
	\end{theorem}
	
	What Ziegler actually shows is that $W(G) \equiv^* X$, which is stronger than the conclusion we state here.  (Recall that $\leq^*$ was the notion of reducibility introduced by Ziegler in \cite[Chapter III.1]{Ziegler} which characterises when every existentially closed supergroup of a fixed finitely generated group must contain another finitely generated group.)  It is not hard to check from the definition that $X \leq^* Y$ implies $X \leq_{\mathrm{T}} Y$ and $X \leq_{\mathrm{e}} Y$.
	
	\subsection{Results on finitely presented extensions}
	
	Higman's Embedding Theorem played a key role in the argument that the existence of an existentially closed group is exactly at the level of $0'$.  In order to relativise that result, we need a relativisation of Higman's Theorem due to Ziegler in \cite{Ziegler}, which we present in this subsection.  We will see that this relativisation rests on enumeration reducibility rather than Turing reducibility.
	
	We start by relativising the notion of finitely presented.
	
	\begin{defn}
		\label{finitely presented extension}
		Let $G$ be a finitely generated group.  We say $F$ is a \emph{finitely presented extension of $G$}, or $F$ is \emph{finitely presented over $G$}, if there is a finite tuple $\tuple{f}$ and a finite collection of words $R(\tuple{g}, \tuple{f})$ such that
		\[ F = \langle \tuple{g}, \tuple{f} \mid W(G) \cup R(\tuple{g}, \tuple{f}) \rangle\]
		and $\tuple{g}$ generates a group isomorphic to $G$ in $F$.
	\end{defn}
	
	In other words, if we know that $G$ is the subgroup of $F$ generated by $\tuple{g}$, we only need finitely many more generators and relations to present $F$.
	
	Ziegler used this concept to prove a relativised analogue of Higman's Embedding Theorem \cite[Theorem II.3.10]{Ziegler}.  The fact that Higman's Theorem relativises suggests that in some sense it is not purely a coincidence that computability and embeddability align, but an incarnation of a deeper relation between the two fields.
	
	\begin{theorem}[Generalised Higman's Embedding Theorem]
		\label{relativisation of higmans theorem}
		Let $G$ and $H$ be finitely generated groups.  Then $W(G) \leq_{\mathrm{e}} W(H)$ iff $G$ embeds in a finitely presented extension $F$ of $H$.
	\end{theorem}
	
	\begin{rem}
		The case when $H = \{e\}$ is just the classical Higman's Embedding Theorem: Repeating a comment we made after Theorem \ref{higman's embedding theorem}, a group with a c.e.\ set of relators will also have a computable set of relators.  To see this, let $\pres{\tuple{g}}{\{R_i(\tuple{g}): i < \omega\} }$ be a presentation for a group with c.e.\ relating set.  Then $\pres{\tuple{g}}{\{R_i(\tuple{g})e^i: i < \omega \}}$ gives an equivalent presentation with a computable set of relators.  
		
		Setting $H = \{e\}$, Theorem \ref{relativisation of higmans theorem} tells us that $W(G) \leq_{\mathrm{e}} 0$ (i.e., $W(G)$ is c.e.) if and only if $G$ embeds in a finitely presented group.  By the argument above, this is a restatement of Higman's Embedding Theorem.
	\end{rem}
	
	\subsection{A degree computed by every existentially closed supergroup}
	Recall that $G$ is a finitely generated group, and we are interested in the complexity of the existentially closed supergroups of $G$.  This section is dedicated to showing that the existence of an existentially closed supergroup of $G$ is exactly at the level of $K_{W(G)}$ (see Definition \ref{def of Gamma and L}).
	
	\begin{notation}
		For the rest of this section, fix $K := K_{W(G)}$.
	\end{notation}
	
	In this subsection, we will show that every existentially closed supergroup of $G$ computes $K$, thereby completing the first half of Theorem \ref{relativisation of existence of ec group is level of hp}.
	
	\begin{theorem}
		\label{ev ec supergroup computes A}
		Every existentially closed supergroup $M$ of $G$ satisfies $\Delta^{qf}(M) \geq_{\mathrm{T}} K$.
	\end{theorem}
	
	We start with a simple observation.
	
	\begin{obs}
		Let $F$ be a finitely presented extension of $G$.  Then $W(F) \equiv_{\mathrm{e}} W(G)$.  
	\end{obs}
	
	In fact, a converse also holds: any set in the enumeration degree of $W(G)$ is Turing-computed by a finitely presented extension of $G$.  
	
	\begin{prop}
		Let $G$ be a finitely generated group.  Then, for any $Y \equiv_{\mathrm{e}} W(G)$, there is a finitely presented extension $F$ of $G$ with $W(F) \geq_{\mathrm{T}} Y$.
	\end{prop}
	
	\begin{proof}
		Let $H$ be a finitely generated group satisfying $W(H) \equiv_{\mathrm{T}} Y$ and $W(H) \equiv_{\mathrm{e}}Y$, as provided by Theorem \ref{every star degree contains a group}.  Then, using Proposition \ref{enumeration of G from presentation}, we have that
		\[ W(H * G) \leq_{\mathrm{e}} W(G) \oplus Y \equiv_{\mathrm{e}} W(G). \]
		
		By the Generalised Higman's Embedding Theorem, Theorem \ref{relativisation of higmans theorem}, $H * G$ embeds in a finitely presented extension $F$ of $G$.  Then $W(F) \geq_{\mathrm{T}} W(G * H)  \equiv_{\mathrm{T}} W(G) \oplus W(H) \geq_{\mathrm{T}} Y$.
	\end{proof}
	
	Recalling from Proposition \ref{L is Turing maximal} that $K \equiv_{\mathrm{e}} W(G)$, we obtain the following corollary.
	
	\begin{cor}
		There is a finitely presented extension $F$ of $G$ with $W(F) \geq_{\mathrm{T}} K$.
	\end{cor}
	
	\begin{proof}
		By Proposition \ref{L is Turing maximal}, $K \equiv_{\mathrm{e}} W(G)$.
	\end{proof}
	
	We now show that any existentially closed supergroup of $G$ must compute the Turing degree of any finitely presented extension of $G$, and so, in particular the Turing degree of $K$.
	
	\begin{theorem}
		Let $G$ be finitely generated, and $M$ an existentially closed group containing $G$.  Then $M \geq_{\mathrm{T}} W(F)$ for any finitely presented extension $F$ of $G$.
	\end{theorem}
	
	The proof is a relativisation of that of Theorem \ref{every ec group computes hp}.
	
	\begin{proof}
		Let $F = \langle \tuple{g}, \tuple{f} \mid W(G, \tuple{g}) \cup R(\tuple{g}, \tuple{f}) \rangle$ be a finitely presented extension of $G$.  For a word $w(\tuple{g}, \tuple{f})$ in $F$, define
		\[ \varphi^F_w(\tuple{g}, \tuple{x}) = \text{``}\bland R(\tuple{g}, \tuple{x}) = 1 \land w(\tuple{g}, \tuple{x}) \neq 1\text{''}. \]
		
		We first claim that $\exists \tuple{x} \varphi_w(\tuple{g}, \tuple{x})$ is consistent with $G = \langle \tuple{g} \rangle$ iff $w(\tuple{g}, \tuple{f}) \neq 1$ in $F$ .
		
		If $w(\tuple{g}, \tuple{f}) \neq 1$, then $F$ witnesses that $\exists \tuple{x} \varphi_w$ is consistent with $G$.  On the other hand, if $\exists \tuple{x} \varphi_w$ is consistent with $G$, then it is solved in some group $H \geq G$.  Let $\tuple{h}$ solve $\varphi_w$ in $H$.  By construction, $\langle \tuple{g}, \tuple{h} \rangle \leq H$ is a quotient of $F$ in which $w \neq 1$.  Thus, $w \neq 1$ in $F$.
		
		Now, let $M$ be an existentially closed supergroup of $G$.  If $\exists \tuple{x} \varphi_w$ is consistent with $G$, then it must be realised in $M$.  Thus, checking if each tuple of elements of $M$ satisfies $\varphi_w$, we get that $W(F)^c$ is $M$-c.e.  On the other hand, $M$ computes $W(G)$ and $W(F)$ is c.e.\ in $W(G)$, so $W(F)$ is $M$-c.e.  Thus, $W(F) \leq_T M$.
	\end{proof}
	
	Thus, we obtain Theorem \ref{ev ec supergroup computes A}, which states that for every existentially closed supergroup $M$ of $G$, $M \geq_{\mathrm{T}} K$.
	
	\subsection{$K$ computes an e.c.\ supergroup of $H$}
	Here we prove the second half of Theorem \ref{relativisation of existence of ec group is level of hp}.
	
	\begin{theorem}
		\label{A computes an ec supergroup}
		Let $G$ be finitely generated, and $K = K_{W(G)}$, as in the previous subsection.  Then $K$ computes an existentially closed supergroup of $G$.
	\end{theorem}
	
	We will prove this in several parts.  We first argue that Theorem \ref{M_0 is computable in hp}, which gives a $0'$-computable existentially closed group, relativises.
	
	\begin{theorem}
		\label{jump of group computes ec supergroup}
		Let $H$ be a finitely generated group.  Then there is an existentially closed supergroup $M_H$ of $H$ that is $W(H)'$-computable.
	\end{theorem}
	
	The construction of $M_H$ is analogous to the construction of $M_0$ in the proof of Theorem \ref{def of M_0} --- indeed, $M_0 = M_{\{e\}}$ ---  where we use the notion of finitely presented extensions instead of finitely presented groups.  We sketch it below.
	
	The following result, noted by Ziegler in \cite[Definition III.4.3]{Ziegler}, gives the skeleton of the desired existentially closed group $M_H$.  We will then show that $M_H$ can be built $W(H)'$-computably using the effective \Fraisse's Theorem (Theorem \ref{rel effective fraisses theorem}).
	
	\begin{lemma}
		\label{ec fraisse class from enum red}
		Let $H$ be a finitely generated group.  Then there is an existentially closed group $M_H$ such that
		\[\Sk(H) = \{ F : W(F) \leq_{\mathrm{e}} W(H) \}.\] 
	\end{lemma}
	
	\begin{proof}[Proof Sketch]
		By Theorem \ref{characterisation of skeleton of ec group}, it suffices to show that $\{ F : W(F) \leq_{\mathrm{e}} W(H) \}$ is an existentially closed \Fraisse\ class (see Definition \ref{properties of skeletons}).  The key observation is that, by Proposition \ref{enumeration of G from presentation}, $W(G) \leq_{\mathrm{e}} R$ for any set of relators $R$ in a presentation of $G$.  Then the lemma follows essentially identically to Theorem \ref{def of M_0}.
	\end{proof}
	
	\begin{lemma}
		\label{M_H is W(H)' computable}
		$M_H$ is $W(H)'$-computable.
	\end{lemma}
	
	\begin{proof}
		The proof is largely a relativisation of Theorem \ref{M_0 is computable in hp}.  Setting
		\[ \cK_H = \{ F : W(F) \leq_{\mathrm{e}} W(H)  \}, \]
		we show that $\cK_H$ has a $W(H)'$-computable representation with $W(H)'$-EC.  (The definition of $\Tdega$-EC can be found in Definition \ref{def of rel fraisse props}.)
		
		Fix a presentation $\pres{\tuple{h}}{R}$ of $H$, where $\tuple{h}$ is an initial segment of $\omega$.  Produce a computable enumeration of presentations $\pres{\tuple{h}, \tuple{f}}{S, R}$ with generating set an initial segment of the natural numbers, and $S$ is a finite of words on $X$ and $\tuple{f}$.  (Note that this is not the same as enumerating the finitely presented extensions of $H$, as $H$ in general will not embed in these groups.)  Now, using $W(H)'$, we can compute the full quantifier-free diagram of each group given by a presentation in this list.  Call the resulting enumeration $\seq{F_i, \tuple{f}_i}{i < \omega}$.
		
		Moreover, for each $i$, denote by $\alpha_i$ the map $H \to F_i$ obtained by sending the generators of $H$ to their image in $F_i$.  We ask $W(H)'$ whether there is a word $w$ and $s \in \omega$ such that $w(\tuple{h}) \neq 1$ in $H$ and $w(f_i(\tuple{h}))$ is the $s$th consequence of the relations of $F_i$.  If the answer is yes, then the elements $\tuple{h}$ in $F_i$ do not generate a copy of $H$ in $F_i$, so $F_i$ is not a finitely presented extension of $H$, and we delete it from the list.
		
		Let $\seq{D^j_n}{n < \omega}$ be an effective list of all the finite sequences of words in the letters $\tuple{f_j}$ where, by convention, $D^j_0 = \{\tuple{f_j}\}$.  Write $F^j_n$ for the subgroup generated by $D^j_n$ in $F_j$, so $F^j_0 = F_j$ and $F^j_n$ embeds in $F_j$ for every $j$ and $n$.  Then $\bK_H = \seq{F^j_n, D^j_n}{\langle j, n \rangle < \omega}$ enumerates all finitely generated groups that embed in a finitely presented extension of $H$.  By the relativisation of Higman's Embedding Theorem, this is equivalent to $\cK_H$.  Moerover, $\bK_H$ is a $W(H)'$-computable enumeration such that the quantifier-free diagrams of the groups are uniformly $\cK_H$-computable.
		
		Thus, it remains to check $W(H)'$-EC, but this follows identically to the argument in Theorem \ref{M_0 is computable in hp}.	
	\end{proof}
	
	\begin{rem}
		This lemma can also be proved without the relativised Higman Embedding Theorem by checking that 
		\[ \cK_H^{\mathrm{T}} = \{ F : W(F)\ \text{has a $W(H)$-computable set of relators} \} \]
		satisfies the conditions of the effective \Fraisse\ Theorem, relativised to $0'$.
		
		In fact, it turns out that $\cK^{\mathrm{T}}_H \supsetneq \cK_H$ in general.  A relativisation of Craig's trick (Remark \ref{Craig}) implies that if $F$ has a presentation which is $\leq_{\mathrm{e}} X$, then it has an $X$-computable one.  
		
		On the other hand, we do not get equality as we can take $H$ a group that is $\mathrm{T}$- and $\mathrm{e}$-equivalent to $\emptyset'$ and $F$ a group that is $\mathrm{T}$- and $\mathrm{e}$-equivalent to $\overline{\emptyset'}$.  Then $H, F \in \cK^{\mathrm{T}}_H$, but only $H \in \cK^{\mathrm{T}}_H$.  
	\end{rem}
	
	We can now obtain the desired $K$-computable existentially closed supergroup of $G$.  
	
	\begin{theorem}
		Let $G$ be finitely generated.  Then $K$ computes an existentially closed supergroup of $G$.
	\end{theorem}
	
	\begin{proof}
		Let $W = W(G)$.  Apply Theorem \ref{enumeration jump inversion} to $W$, which gives a total $X \geq_e W$ with $J_e(X) \equiv_{\mathrm{e}} J_e(W)$.  Since both of these sets are total, we get $J_e(X) \equiv_{\mathrm{T}} J_e(W)$.  Moreover, since $X$ is total, $J_e(X) \equiv_{\mathrm{T}} X'$ by Proposition \ref{jumps of total sets a T equivalent} and $J_e(W) \equiv_{\mathrm{T}} K$ from the definition of $K$.  Thus, $X' \equiv_{\mathrm{T}} K$.
		
		By Theorem \ref{every star degree contains a group}, there is a finitely generated $H$ with $W(H) \equiv_{\mathrm{e}} X$ and $W(H) \equiv_{\mathrm{T}} X$.  By Theorem \ref{jump of group computes ec supergroup}, we can obtain an existentially closed group that is $W(H)'$-computable, and hence $X'$-computable and $K$-computable.
		
		We just need to check that $G$ embeds in $H$.  But $W(G) \leq_{\mathrm{e}} X$, so $G \in \cK_H$, as required. 
	\end{proof}
	
	This completes the proof of the relativisation of Theorem \ref{A computes an ec supergroup}, and hence we have shown that $K$ is the optimal degree of an existentially closed supergroup of $G$.
	
	\subsection{Degree spectrum of existentially closed groups}
	As an application of the work above, we characterise the Turing degrees which are degrees of existentially closed groups.  
	
	Recall that, for a model $M$, the degree of $\Delta^{\mathrm{qf}}(M)$ (or, for simplicity, the degree of $M$) is the minimal Turing degree, if one exists, that computes the constants, functions, and relations of $M$, relative to some enumeration.  The following definition captures all such minimal degrees for existentially closed groups $M$.
	
	\begin{defn}
		Let 
		\[\operatorname{DegSpec}(ECGroup) = \{ \Tdega : \deg(\Delta^{\mathrm{qf}}(M)) = \Tdega \text{ for some countable e.c.\ group } M \}\] 
		be the \emph{degree spectrum} of existentially closed groups. 
	\end{defn} 
	
	We can characterise such degrees.
	
	\begin{theorem}
		\label{degree spectrum of ec groups}
		The degree spectrum of existentially closed groups is given by $$\operatorname{DegSpec}(ECGroup) = \{\Tdega : \Tdega \geq_{\mathrm{T}}0'\}.$$
	\end{theorem}

	\begin{proof}
		By Theorem \ref{existence of ec groups is at level of hp}, we already have 
		\[\mathrm{DegSpec}(ECGroup) \subseteq \{\Tdega : \Tdega \geq_{\mathrm{T}}0'\}.\]  
		
		On the other hand, let $\Tdega \geq_{\mathrm{T}} 0'$, and let $\Tdegb$ be such that $\Tdegb' \equiv_{\mathrm{T}} \Tdega$.  Let $A \in \Tdega$ and $B \in \Tdegb$ be total.  Applying Proposition \ref{jumps of total sets a T equivalent}, we get that $J_e(B) \equiv_{\mathrm{T}} K_B \equiv_1 B' \equiv_{\mathrm{T}} \Tdega$.  
		
		Thus, for every $\Tdega \geq 0'$, $\Tdega \equiv_{\mathrm{T}} K_B$ for some $B$.  Taking a group $G$ such that $W(G) \equiv_{\mathrm{e}} B$ and $W(G) \equiv_{\mathrm{T}} B$ and applying Theorem \ref{relativisation of existence of ec group is level of hp} we get the result.
	\end{proof}
	
	\section{Existentially closed groups with ``uncomplicated'' subgroups}
	In this section, we delve deeper into our earlier result that the existence of existentially closed groups is at the level of the halting problem (Theorem \ref{existence of ec groups is at level of hp}).  We show that this complexity arises only from having to determine which systems of equations and inequations to realise, and not from having to embed ``complicated'' finitely generated subgroups.  
	
	The key idea is that, once we know an existential formula $\exists \tuple{x} \varphi(\tuple{x}, \tuple{g})$ needs to be realised (using $0'$), actually building a finitely generated group that realises it can be done by finding an element in some \emph{$\Pi^0_1$-class} (See Definition \ref{def: pi01 class}).  
	
	This provides a new, algebraic characterisation of an important class of Turing degrees called the \emph{PA degrees} (Theorem \ref{char of PA degrees}).  The PA degrees will be defined in the following subsection.
	
	\subsection{Background on Scott sets}
	\label{section: background on scott sets}
	A key tool in our proof will be the notion of a \emph{Scott set}, introduced by Dana Scott in \cite{Scott}.  Before we prove the main result of this section, we collect relevant results on Scott sets.
	
	\begin{defn}
		\label{def: pi01 class}
		Let $X \subseteq \omega$.  A \emph{$\Pi^{0,X}_1$-class} $\cC$ is a set of the form
		\[ \cC  = \{ \sigma \in 2^{\omega} : \forall n R(\sigma \upharpoonright n) \} \]
		where $R$ is an $X$-computable relation.
	\end{defn}
	
	\begin{defn}
		A \emph{Scott set} is a collection $\cS$ of subsets of $\omega$ satisfying all of the following:
		\begin{itemize}
			\item If $X \in \cS$ and $Y \leq_{\mathrm{T}}X$, then $Y \in \cS$.
			\item If $X$, $Y \in \cS$, then $X \oplus Y \in \cS$.
			\item If $\mathcal{C}$ is a nonempty $\Pi^{0,X}_1$ for some $X \in \cS$, then $\cS$ contains an element of $\cC$. \qedhere
		\end{itemize}
	\end{defn}
	
	Scott sets were introduced to study models of Peano arithmetic.  They are intimately connected with the computability-theoretic notion of \emph{PA degrees}.
	
	\begin{defn}
		\label{PA deg-def}
		A \emph{PA degree} is a Turing degree that computes a complete extension of Peano arithmetic.
	\end{defn}
	
	There are many equivalent formulations of PA degrees.  The following will be most relevant for our discussion.
	
	\begin{prop}
		For a Turing degree $\Tdega$, the following are equivalent:
		\begin{enumerate}
			\item $\Tdega$ is a PA degree.
			\item $\Tdega$ computes an element in every nonempty $\Pi^0_1$-class.
		\end{enumerate}
	\end{prop}
	
	The key result connecting PA degrees and Scott sets is the following, in which we use $\Phi^X_e$ to denote the $e$th partial $X$-computable function with image in $\{0,1\}$.  If $\Phi^X_e$ is total, we identify it with the set for which it is the characteristic function.  The second part is due to \cite{Macintyre-Marker}.
	
	\begin{theorem}
		\label{Macintyre--Marker}
		Every Scott set contains a PA degree.  
		
		On the other hand, every Scott set $\cS$ is \emph{effectively enumerated} by a PA degree $\Tdega$; i.e., there is an $\Tdega$-computable enumeration $\seq{X_i}{i < \omega}$ and $\Tdega$-computable functions $\alpha$, $\beta$, and $\gamma$ such that: 
		\begin{itemize}
			\item If $\Phi^{X_i}_e = Y$, then $Y = X_{\alpha(i, e)}$
			\item If $X_i \oplus X_j = X_{\beta(i, j)}$
			\item If $\mathcal{C}$ is a nonempty $\Pi^{0, X_i}_1$ class, then $X_{\gamma(i)}$ is an element of $\mathcal{C}$.
		\end{itemize}
	\end{theorem}
	
	In other words, for every Scott set $\cS$, there is a PA degree $\Tdega$ computing every element of $\cS$ and also the functions witnessing that $\cS$ is a Scott set.
	
	\subsection{Existentially closed groups from Scott sets}
	Given a Scott set $\cS$, we now construct an existentially closed group $M_{\cS}$ whose finitely generated subgroups ``span'' $\cS$.  This will give a number of corollaries, including a new characterisation of the PA degrees.
	
	\begin{theorem}
		\label{every scott set gives an ec group, original}
		Let $\cS$ be a Scott set.  Then there is an existentially closed group $M_{\cS}$ such that the Turing degrees in $\Sk(M)$ are exactly the Turing degrees of $\cS$.
	\end{theorem}
	
	\begin{rem}
		We have already seen that the converse does not hold: there is an existentially closed group the Turing degrees of whose skeleton do not form a Scott set.  Consider $M_0$, the existentially closed group whose skeleton consisted of all finitely generated, computably presentable groups.  The set of Turing degrees of its skeleton contains a maximal element; namely any computably presentable group with a $0'$-computable word problem.  However, Scott sets never have maximal degrees.
		
		This raises a natural question: can one computability-theoretically characterise the subsets of the Turing degrees which arise as the skeletons of existentially closed groups?  We know, for example, that such subsets are not necessarily downwards-closed --- as the c.e.\ degrees comprise the skeleton of $M_0$ --- but that it is closed under joins, by the JEP.
	\end{rem}

	\begin{proof}[Proof of Theorem \ref{every scott set gives an ec group, original}]
		Let $\cK_{\cS} = \{ G : \text{a finitely generated group and } W(G) \in \cS \}$.  By Theorem \ref{characterisation of skeleton of ec group}, it suffices to show that $\cK_{\cS}$ is an existentially closed $\Fraisse$ class with HP, JEP, and AP.
		
		It is clear that $\cK_{\cS}$ satisfies HP, JEP, and AP since taking finitely generated subgroups, free products, and free products with amalgamation respectively produce groups whose word problems are no greater than the joins of the original groups.  Thus, it suffices to show that it is existentially closed.  Let $G \in \cK_{\cS}$, let $\varphi(\tuple{x}, \tuple{g})$ be a quantifier-free formula with parameters from $G$ that is satisfied in some supergroup of $G$.
		
		By Proposition \ref{no disjunctions for ec}, we may assume that $\langle \tuple{g} \rangle$ generates $G$ and that $\varphi$ is equivalent to ``$ \bland_{i \in I} w_i(\tuple{x}, \tuple{g}) = 1 \land \bland_{j \in J} w_j(\tuple{x}, \tuple{g}) \neq 1$'' for some finite sets $I$ and $J$.  
		
		Let $\tuple{h}$ be a tuple of new constants of length $|\tuple{x}|$.  Fix an ordering on the words in $\tuple{g}, \tuple{h}$ and identify $\sigma \in 2^{\omega}$ with the theory 
		\[ \hat{T} = \{ (w_n(\tuple{g}, \tuple{h}) = 1)^{\sigma(n)} \} \]
		where $\psi^0 = \neg \psi$ and $\psi^1 = \psi$.  Say a finite theory is \emph{$n$-consistent} if there is no proof of a contradiction from it of length $< n$.  Define the $\Pi^{0, W(G)}_1$ class
		\[ \mathcal{C} = \{ \hat{T} :  \forall n\ \forall t \leq n \left(  \hat{T} \upharpoonright t\ \text{is $n$-consistent} \right) \}. \]
		
		The elements of $\mathcal{C}$ correspond to quantifier-free diagrams of groups generated by $\tuple{g}$ and $\tuple{h}$ which solve $\varphi(\tuple{x}, \tuple{g})$.  Since $\varphi(\tuple{x}, \tuple{g})$ was assumed to be solvable over $G$, it follows that $\mathcal{C} \neq \emptyset$.  Thus, by definition of a Scott set, there is some $\hat{T} \in \cS$ which is also in $\mathcal{C}$.   Hence, $\cS$ contains the quantifier-free diagram of a group $H$ generated by $\tuple{g}, \tuple{h}$ which contains $G$ and solves $\varphi$.  This finishes the proof.
	\end{proof}
	
	From the above result and the equivalence between Scott sets and PA degrees, we obtain the following theorem highlighting the relationship between skeletons of existentially closed groups and PA degrees.
	
	\MainTwo
	
	\begin{proof}
		Let $\Tdega$ be a PA degree.  Then by Theorem \ref{Macintyre--Marker}, $\Tdega$ computes a Scott set $\cS$.  By Theorem \ref{every scott set gives an ec group, original}, there is an existentially closed group $M_{\cS}$ such that every finitely generated $G \leq M_{\cS}$ has $W(G) \in \cS$ and hence $W(G) \leq_{\mathrm{T}}\Tdega$.
	\end{proof}
	
	From this, we obtain a new characterisation of the PA degrees.
	
	\MainThree
	
	\begin{proof}
		Let $\Tdega$ be a PA degree.  Then Theorem \ref{every scott set gives an ec group} gives an existentially closed group $M_{\Tdega}$ such that every finitely generated subgroup of $M_{\Tdega}$ is computable in $\Tdega$.  
		
		On the other hand, suppose $\Tdega$ is a Turing degree and $M$ is an existentially closed group such that $\Tdega$ computes the word problem of every finitely generated subgroup of $M$.  Following the proof of \cite[Theorem II.2.11]{Ziegler}, we see that $M$ has a finitely generated subgroup whose word problem has PA degree.  Since PA degrees are upwards closed, $\Tdega$ must be a PA degree.
	\end{proof}
	
	\begin{rem}
		Theorem \ref{char of PA degrees} and Theorem \ref{relativisation of existence of ec group is level of hp} hint at the reverse mathematical strength of building existentially closed groups.  In particular, we conjecture that the existence of an existentially closed \Fraisse\ class of groups is equivalent to $\WKL$ over $\RCA$ and the existence of an existentially closed supergroup of a group is equivalent to $\ACA$ over $\RCA$.  
		
		However, as suggested by the machinery needed to prove Theorem \ref{relativisation of existence of ec group is level of hp}, a number of interesting subtleties arise, which the author hopes to address in a later paper.
	\end{rem}
	
	\begin{rem}
		In \cite{Miller}, Miller constructs a finitely presented group such that every nontrivial quotient (i.e., quotient that is not equal to the trivial group) has an unsolvable word problem.  He notes that the computably presented quotients have degree $0'$, but that, in general, the quotients may not even have c.e.\ word problem.  
		
		Our result shows that Miller's theorem cannot be improved to guarantee that every nontrivial quotient is in a cone above some fixed PA degree.  Suppose there were such a finitely presented group $F = \pres{\tuple{f}}{R_0, \ldots, R_n}$ such that every nontrivial quotient is strictly above some PA degree $\Tdega$.  Then, since $F$ is finitely presented, every existentially closed group contains a nontrivial quotient of $F$.  But no such quotient embeds in $M_{\Tdega}$, a contradiction.
		
		In fact, using the existentially closed group $M_{\Tdega}$ from Corollary \ref{ec group with every fp subgroup computable} below, a similar argument shows that there is no finitely presented $F$ such that every nontrivial quotient of $F$ has c.e.\ degree.
	\end{rem}

	\subsection{Applications}
	In this subsection, we explore some applications of the equivalence between PA degrees and skeletons of existentially closed groups developed above.  
	
	Our first application highlights that the degree of the finitely generated groups in the skeleton may in general be ``quite far'' from the degree of the whole group.
	
	\begin{cor}
		There is an existentially closed group all of whose finitely generated subgroups have low degree --- and in particular, all have degree strictly below $0'$.
	\end{cor}
	
	\begin{proof}
		There is a low PA degree (see \cite[Chapter VIII.2]{Simpson}).
	\end{proof}
	
	Moreover, we can use other properties of PA degrees to get results about the structure of the class of existentially closed groups.
	
	\begin{cor}
		\label{ec group with every fp subgroup computable}
		There is an existentially closed group $M$ such that each finitely presentable subgroup of $M$ has a solvable word problem.
	\end{cor}
	
	\begin{proof}
		There is a hyperimmune-free PA degree $\Tdega$ (see, for example, \cite[Chapters 5 and 9]{Soare}). Let $M_{\Tdega}$ be the existentially closed group from Theorem \ref{char of PA degrees}.  Because hyperimmune-free degrees are downwards closed and do not contain nonzero c.e.\ sets, it follows that the only groups a with c.e.\ word problem that embed in $M_{\Tdega}$ have a solvable word problem.  But the word problem of every finitely presented group can be computably enumerated, and so has c.e.\ degree.
	\end{proof}		
	
	Our last corollary is Theorem 4.1.4 from \cite{Hodges}.  Recall the following definition from the introduction.
	
	\begin{defn}
		Let $T$ be a theory.  A complete, quantifier-free type $\Phi(\overline{x})$ is \emph{$\exists_1$-isolated} (with respect to $T$) if there is an $\exists_1$-formula $\varphi(\overline{x})$ such that $\Phi$ is the unique quantifier-free type consistent with $\varphi$ over $T$.
	\end{defn}
	
	\begin{cor}
		\label{two relatively atomic ec groups below PAs}
		There are two existentially closed groups $M$ and $N$ such that if $G$ is a finitely generated group embedding into both, then $G$ has solvable word problem; equivalently, by Theorem \ref{macintyre-neumann-rips}, $G$ is $\exists_1$-isolated.
	\end{cor}
	
	\begin{rem}
		The $M$ and $N$ in this corollary are necessarily distinct by Theorem \ref{char of PA degrees}.  This gives a different proof that, for any group with unsolvable word problem, there is an existentially closed group into which it does not embed.
	\end{rem}
	
	\begin{proof}[Proof of Corollary \ref{two relatively atomic ec groups below PAs}]
		There is a minimal pair of PA degrees.  (This follows from \cite[Corollary VIII.2.6]{Simpson} and the theorem, due from \cite{Scott} that $\omega$-models of $\WKL$ correspond to what we call Scott sets, and thus to PA degrees, by the discussion at the beginning of the section.)  Thus the only degree that appears in both is $0$.
	\end{proof}
	
	This is a different proof from Hodges', who used a variant of model-theoretic forcing to build $M$ and $N$.
	
	\section{Building relatively atomic existentially closed groups}
	\label{builing relatively atomic ec groups-section}
	The starting point for this section is Corollary \ref{two relatively atomic ec groups below PAs}, which states that there are two existentially closed groups $M$ and $N$ such that any finitely generated group that embeds in both has a solvable word problem.  Since, by Theorem \ref{macintyre-neumann-rips}, the only finitely generated groups that are contained in every existentially closed group are those with solvable word problem, $M$ and $N$ are, in a quantifiable sense, ``as different as possible''.  
	
	As noted in the previous section, Corollary \ref{two relatively atomic ec groups below PAs} has a model-theoretic incarnation: there are existentially closed groups $M$ and $N$ such that if $\Phi$ is a complete, quantifier-free type realised in both, then $\Phi$ is $\exists_1$-isolated.
	
	\begin{defn}
		Two structures $M$ and $N$ in the same language are \emph{relatively quantifier-free atomic} (or simply, \emph{relatively atomic}) if any quantifier-free type realised in both of them is $\exists_1$-isolated.
	\end{defn}

	In this section, we investigate the computability-theoretic strength of constructing relatively atomic existentially closed groups, and prove an upper bound (Theorem \ref{rel atomic groups below ce}).
	
	\begin{theorem}
		\label{relatively atomic ec groups below 0' ce}
		Let $N$ be a $0'$-computable existentially closed group and let $A$ be $0'$-c.e., $A \gneqq_{\mathrm{T}} 0'$.  Then $A$ computes the quantifier-free diagram of an existentially closed group $M$ such that $M$ and $N$ are relatively atomic.
	\end{theorem}
	
	\MainFour
	
	By Theorem \ref{existence of ec groups is at level of hp}, $M$ and $N$ must both compute $0'$.  Thus, this theorem leaves open the following natural question.
	
	\begin{quest}
		Is there a pair of relatively atomic existentially closed groups both of which are $0'$-computable?
	\end{quest}
	
	The proof of Theorem \ref{relatively atomic ec groups below 0' ce} is a $0'$-c.e.-permitting finite injury argument.  In it, we dovetail three constructions: ensuring the final theory gives an existentially closed group, that the group is relatively atomic with $N$, and ensuring the final theory is complete.  The only part that, on the face of it, requires a more powerful oracle than $0'$ is ensuring the constructed model $M$ is relatively atomic with $N$.  
	
	To guarantee relative atomicity, we meet the requirements 
	\[ R_{\langle \tuple{c}, \tuple{d} \rangle}: \text{``If $\qftp^{M}(\tuple{c}) = \qftp^{N}(\tuple{d})$, then this type is $\exists_1$-isolated''} \]
	for every $\tuple{c}$ in the domain of $M$ and $\tuple{d}$ in the domain of $N$.
	
	To satisfy $R_{\langle \tuple{c}, \tuple{d} \rangle}$, we need to find a quantifier-free formula $\varphi_i$ with parameters from the domain of $M$ which guarantees $\qftp^{M}(\tuple{c}) \neq \qftp^{N}(\tuple{d})$.  However, determining whether a given such $\varphi_i$ can be used to satisfy a given $R_{\langle \tuple{c}, \tuple{d} \rangle}$ is seemingly a $0''$ question:
	
	\emph{Does there exist a consequence $\theta(\tuple{c})$ of $\varphi_i \land T_s$ or $\neg\varphi_i \land T_s$ such that $N \models \neg\theta(\tuple{d})$?}
	
	Thus we cannot, if we wish for the construction to remain below $0''$, apply the na\"ive strategy of inductively checking if $\varphi_i$ and $\neg\varphi_i$ are both consistent and, if so, choosing the highest-priority $R_{\langle \tuple{c}, \tuple{d} \rangle}$ they can satisfy.
	
	Instead, the requirement $R_{\langle \tuple{c}, \tuple{d} \rangle}$ will ask for permission from some $0'$-c.e.\ set $A \gneqq_{\mathrm{T}} 0'$ to add a sentence $\varphi(\tuple{c})$ to the theory which witnesses that $\tuple{c}$ and $\tuple{d}$ have different quantifier-free types.  We will thus guarantee that the final theory is $A$-computable.

	\begin{proof}[Proof of Theorem \ref{relatively atomic ec groups below 0' ce}]
		Let $T$ be the theory of groups in the language $L = \{1, \cdot, \bullet^{-1} \}$.  Let $A$ be as in the theorem with $0'$-enumeration $\{ a_s : s < \omega\}$.  For simplicity, we will assume that, for $s \equiv_3 1$, $a_s = a_{s + 1} = a_{s + 2}$.  Let $C$ be a countable set of new constant symbols, and let $T^N = \Delta^{\mathrm{qf}}(N)$ via some effective assignment of the constants of $C$ to $N$. Let $\seq{\varphi_m(\tuple{x})}{m < \omega}$ enumerate the quantifier-free formulas and $\seq{\psi_i}{i < \omega}$ the positive, quantifier-free sentences in $L(C)$.  We will build a theory $T^M$ satisfying the requirements below.
		
		\begin{itemize}
			\item $\mathrm{Group}$: ``The structure determined by $T^M$ models the axioms of group theory.''
			\item $EC_m$: ``If $\exists \tuple{x} \varphi_m(\tuple{x})$ is consistent with $T^M$, then $T^M \models \varphi_m(\tuple{c})$ for some set of constants $\tuple{c}$.''
			\item $R_{\langle \tuple{c}, \tuple{d} \rangle}$: ``If $\qftp^{M}(\tuple{c}) = \qftp^{N}(\tuple{d})$, then this type is $\exists_1$-isolated.''
		\end{itemize}
		
		(Here, $\qftp^M(\tuple{c})$ and $\qftp^N(\tuple{d})$ are the quantifier-free types of $\tuple{c}$ and $\tuple{d}$ in $T^M$ and $T^N$ respectively.  As we will guarantee that $T^M$ is quantifier-free-complete, these will be complete quantifier-free types.)
		
		We use the ordering $\mathrm{Group}, R_0, EC_0, R_1, EC_1, \ldots$ of the requirements.  Throughout, we will refer to the $k$th element of this list as the ``$k$th requirement'', starting with $\mathrm{Group}$ as the $0$th requirement.  Thus, $R_m$ is the ($2m+1$)st requirement and $EC_m$ is the ($2m+2$)nd requirement.  In addition, we will define along the way an auxiliary function $r( k, s)$ that ``restrains'' how the $k$th requirement is allowed to act at stage $s$.\\
		
		\noindent\textit{Construction:} We will build $T^M$ inductively.  We will variously think of $T^M_s$ as a set of sentences $\{ \psi_i : i \in I \} \cup \{ \neg \psi_j : j \in J \}$ or as the corresponding partial indicator function which is $1$ on $I$, $0$ on $J$, and undefined elsewhere.  We establish the following convention. 
		
		\begin{defn}
			Set
			\[ \lceil T^M_s \rceil = \max \{ i : T^M_s(i)\ \text{is defined} \} \qedhere \]
		\end{defn}
		
		In addition, throughout we will write $\hat{T}^M_s$ for $T^M_s$ together with the axioms of group theory.
		
		\textit{Stage 0}: $T^M_0 = \{ \text{``$c_i \cdot e = e \cdot c_i = c_i$'';\ ``$c_i \cdot c_i^{-1} = e$'';\ ``$(c_i \cdot c_j) \cdot c_k = c_i \cdot (c_j \cdot c_k)$''}: c_i, c_j, c_k \in C \}$.  This will ensure that $M$ is a group\footnote{We will guarantee that in our final model every element is labelled by a constant: in the process of checking that the model satisfies existential closure, for any word $w$ in $C$, we will choose a constant $c \in C$ to realise the consistent existential $\exists x (w = x)$.  Thus this does in fact ensure that the final structure is a group}.  While we will not explicitly mention it, we will assume that these sentences are in $T^M_s$ for every $s$.  It will be clear that this will not affect the construction.  
		
		Initialise $r(k, 0) = r(k, 1) = 0$ for every $k < \omega$.  
		
		\textit{Stage $s \equiv_3 1$:}  In this stage, we take another step towards ensuring that $M$ is existentially closed.  Let $m$ be minimal such that $r(2m+2, s) = 0$.  Using $0'$, check the consistency of $\exists \tuple{x} \varphi_m(\tuple{x})$ with $\hat{T}^M_s$.  If it is consistent, set $T^M_{s+1} = T^M_s \cup \{ \varphi_m(\tuple{c}) \}$ for constants $\tuple{c}$ not used since stage 0 and such that $\varphi_m(\tuple{c}) > \max\{ r(k, s) : k < 2m+2 \}$.  If it is not consistent, set $T^M_{s+ 1} = T^M_s$.
		
		Define $r(2m+2, s+1) = \lceil T^M_{s+1} \rceil$, $r(k, s + 1) = 0$ for $k > 2m+2$, and $r(k, s+1) = r(k, s)$ for $k < 2m+2$.
		
		\textit{Stage $s \equiv_3 2$:} In this stage, we take another step towards ensuring that $M$ is relatively atomic to $N$.  We start with a definition which captures what we are looking for to satisfy a $R_m$-requirement. 
		
		\begin{defn}			
			For a finite theory $\tuple{p}$ in $L(C)$, a requirement $R = R_{\langle \tuple{c}, \tuple{d}\rangle}$, and $s < \omega$, say a quantifier-free sentence \emph{$\psi$ $s$-splits $R$ mod $\tuple{p}$} if the constants of $\psi$ are among the $\tuple{c}$, and $\tuple{p} \cup \{ \psi \}$ and $\tuple{p} \cup \{ \neg \psi \}$ are both consistent.  We call $\psi$ the \emph{witness} to the splitting of $R$.
		\end{defn}
		
		For each $R_m$, $m \leq s$, with $r(2m + 1, s) = 0$, set $\ell_m = \max\left\{a_s, \max\{r(k, s) : k < 2m + 1\}\right\}$ and $\tuple{p}_m = \hat{T}^s_M \upharpoonright \ell_m$.  For each $\psi_i$ with $\ell_m \leq i \leq \ell_m + s$, check if $\psi_i$ $s$-splits $R_m \! \mod \tuple{p}_m$.  If no splitting is found, let $T^M_{s + 1} = T^M_s$ and $r(k, s + 1) = r(k, s)$.
		
		If a splitting is found, let $m = \langle \tuple{c}, \tuple{d} \rangle$ be minimal such that $R_m$ is $s$-split and let $\psi$ witness the splitting.  If $\psi(\tuple{d} / \tuple{c}) \in T^N$, set:
		\begin{align*}
			T^M_{s + 1} &= \tuple{p}_m \cup \{ \neg \psi \}.
		\end{align*}
		Otherwise, $\neg \psi(\tuple{c} / \tuple{d}) \in T^N$, so set:
		\begin{align*}
			T^M_{s + 1} &= \tuple{p}_m \cup \{ \psi \}.			
		\end{align*}
		Set $r(k, s+1) = 0$ for $k > 2m + 1$.  Set $r(2m +1, s+1) = \lceil T^M_{s+1} \rceil$ and $r(k, s+1) = r(k, s)$ otherwise.
		
		\textit{Stage $s \equiv_3 0$:} Let $i$ be minimal such that $T_s^M(i)$ is undefined.  Set $T^M_{s+1} = T^M_s \cup \{ \psi_i \}$ if this is consistent and $T^M_{s + 1} = T^M_s \cup \{ \neg \psi_i \}$ otherwise.  Set $r(k, s+1) = r(k, s)$ for every $k$.\\
		
		This finishes the construction.  Set $T^M = \lim_s T^M_s$.\\
		
		\noindent\textit{Verification:}  The stages $s \equiv_3 2$ ensure that $T^M$ is quantifier-free-complete.  Thus we may consider the structure $M$ that $T^M$ determines.  Notice that many elements of $C$ may refer to the same element of $M$; this will not affect the computability of $M$ as these equalities are contained in $T^M$.  Stage 0 guarantees that $M$ is a group.  We show that $M$ is:
		\begin{itemize}
			\item existentially closed
			\item computable in $A$
			\item relatively atomic with $N$
		\end{itemize}
		
		Say the $k$th requirement \emph{acts at stage $s$} if $r(k, s) = 0$ and $r(k, s+1) \neq 0$ and that it is \emph{satisfied at stage $s$} if $r(k, s) > 0$.  Note that no requirement acts infinitely often: once a requirement has acted, it remains satisfied unless a higher priority requirement acts.
		
		Thus, to check that $EC_m$ holds, consider the minimal $s_m$ such that all higher priority requirements do not act at a later stage.  If $EC_m$ is already satisfied, it will remain satisfied for the rest of the construction as no higher priority requirements act again.
		
		If $EC_m$ is not satisfied, the least $s \geq s_m$ with $s \equiv_3 2$ will ensure that $EC_m$ is satisfied, and, again, it will remain so for the rest of the construction.  
		
		This shows that $M$ is existentially closed.
		
		To compute $M$ from $A$, let $\psi_i$ be a positive, quantifier-free sentence in $L(C)$.  With an oracle for $A$, find a stage $s$ such that $A_s \upharpoonright i+1  = A \upharpoonright i+1$.  From that point on, $T^M_s \upharpoonright i + 1$ changes only on stages $s \equiv_3 2$ and $s \equiv_3 0$ (both of which are $0'$-computable) and will never become undefined or change the value of $j \leq i$ which is already defined.  Thus, $T_M^s(i)$ will eventually be defined and its value will never change.  This gives an $A$-computable way of determining if $\psi_i \in T^M$.
		
		Finally, we show that $M$ and $N$ are relatively atomic.  Suppose not, and consider the minimal $m = \langle \tuple{c}, \tuple{d} \rangle$ such that $\qftp^M(\tuple{c}) = \qftp^N(\tuple{d})$, but they are not $\exists_1$-isolated.  We will show this implies $A \leq_T 0'$, a contradiction.  Let $s$ be a stage by which all requirements of lower priority than $R_m$ have stopped acting.  Let $\tuple{p}_t = \max\{ r(k, t) : k < m \}$.  By assumption that lower priority requirements have stopped acting by stage $s$, $\tuple{p}_t = \tuple{p}_s =: \tuple{p}$ for all $t \geq s$.  Moreover, $\tuple{p} \subseteq T^M$, since none of the remaining, unsatisfied requirements are high enough priority to change $\tuple{p}$.  Thus, $R_m$ splits mod $\tuple{p}$.  To determine if $x \in A$, let $i > x$ with $\psi_i$ witnessing the splitting of $R \mod \tuple{p}$.  Since $R_m$ is not satisfied, at a stage $t \equiv_3 2$ with $t \geq \max\{s, i\}$, it will see the splitting $\mod \tuple{p}$ but not get permission to act on it.  Thus, $A_t \upharpoonright i = A \upharpoonright i$.  Hence, we can determine if $x \in A$, contradicting our choice of $A$, so $M$ and $N$ must be relatively atomic.
	\end{proof}
	
	The proof can be easily modified to give the following relativisation.
	
	\begin{theorem}
		Let $N$ be an $\Tdega$-computable existentially closed group, and let $B$ be $\Tdega$-c.e., $B \gneqq_{\mathrm{T}} \Tdega$.  Then $B$ computes the quantifier-free diagram of an existentially closed group $M$ that is relatively atomic with $N$.
	\end{theorem}
	
	\vfill \pagebreak
	\bibliographystyle{alpha}
	\bibliography{on-effective-constructions.bib}
	
\end{document}